\documentclass{article} \usepackage[backend=bibtex,style=numeric,sorting=none]{biblatex}
\usepackage[parfill]{parskip}
\usepackage{graphicx}
\usepackage{subcaption}
\usepackage{standalone}
\usepackage{caption}
\usepackage{amsmath}
\usepackage{amssymb}
\usepackage{amsthm}
\usepackage{enumitem}
\usepackage{mathtools}
\usepackage{amssymb}
\usepackage{comment}
\usepackage{float}
\usepackage{thmtools}
\usepackage{algorithm}
\usepackage{algpseudocode}
\usepackage{hyperref}
\usepackage{cleveref}
\usepackage{mathdots}
\usepackage{microtype}
\usepackage{xcolor}
\usepackage[a4paper, left=1in, right=1in, top=1in, bottom=1in]{geometry}
\usepackage[most]{tcolorbox}

\makeatletter
\newcommand{\namedlabel}[2]{%
  \protected@edef\@currentlabel{#2}% Set standard label text
  \protected@edef\@currentlabelname{#2}% Set nameref text
  \label{#1}%
}
\makeatother

\newtheorem{theorem}{Theorem}
\newtheorem*{theorem*}{Theorem}
\newtheorem{lemma}{Lemma}
\newtheorem{corollary}{Corollary}
\newtheorem{remark}{Remark}
\newtheorem{definition}{Definition}
\newtheorem{example}{Example}
\newtheorem{proposition}{Proposition}
\newtheorem{conjecture}{Conjecture}

\tcbuselibrary{skins}
\tcbuselibrary{theorems}

\DeclareDelimFormat{multicitedelim}{\addsemicolon\space}

\title{An Infinite Family of Apéry-like Recurrences for Euler's Constant}

\author{Rotem Kalisch \and Tali Monderer \and Elyasheev Leibtag \and Shachar Weinbaum \and Michael Shalyt \and Ido Kaminer}
\date{}
\begin{document}

\maketitle

\section{Abstract}

Euler's constant $\gamma$, also known as the Euler–Mascheroni constant, remains one of the most elusive constants in mathematics.
Despite extensive study, its irrationality is still unproven, and Ap\'ery-like approximations of it are remarkably scarce.

We present a novel identity that expresses $\gamma$ in terms of a Meijer G-function and its first few derivatives. This identity is then generalized and used to systematically generate \textit{infinitely many} Apéry-like approximations of $\gamma$ using a Conservative Matrix Field construction based on the Meijer G-function.

\section{Introduction}

The study of Diophantine approximations has long been motivated by the desire to understand the arithmetic nature of fundamental constants.
Among the various tools used in related research, \emph{Ap\'ery-like approximations} (sequences generated by a linear recurrence with polynomial coefficients, also known as \emph{holonomic sequences}, \emph{D-Finite sequences} or \emph{P-Recursive sequences} \cite{stanley1980differentiably, lipshitz1989d, zeilberger1990holonomic}) have dominated the field in recent years.
Their appeal stems from their connection to linear differential equations, for which many results regarding generating functions and asymptotics have been discovered.

A landmark example is Apéry’s proof of the irrationality of $\zeta(3)$, which involved a fast converging holonomic sequence \cite{apery1979irrationalite, van1979proof}.
Inspired by this success, researchers have sought similar sequences that converge to other constants of interest \cite{ball2001irrationalite, zudilin2002ap, guillera2006generators, raayoni2019ramanujan, calegari2024linear}.
Indeed, for fundamental constants such as $e$, $\pi$, $\zeta(2)$, and $\zeta(3)$, many such sequences are known and their irrationality has been rigorously proven.

Despite a large effort by the mathematical community, Euler's constant $\gamma$ \cite{euler1740progressionibus, euler1785numero, mascheroni1790adnotationes} remains notoriously elusive.
Not only is there no proof of its irrationality, but there are also very few known Ap\'ery-like recurrences capable of approximating it in a similar manner \cite{aptekarev2009linear, pilehrood2013continued, rivoal2009rational}.

A foundational result in the study of Euler's constant $\gamma$, also involving the Gompertz constant $\delta$, was established by Shidlovskii \cite{shidlovskii1962, mahler2006lectures}, who studied algebraic independence of E-function values and, as a corollary, proved that at least one of these constants must be irrational. 
This result was improved by Rivoal, who showed that at least one of them is transcendental \cite{rivoal2012arithmetic}.

A significant advancement by Aptekarev \cite{aptekarev2007volume, aptekarev2009linear} yielded the first Apéry-like approximation for \(\gamma\) in the form of a third-order recurrence. The same recurrence also approximates \(\delta\) using a different ratio of two linearly independent solutions.
% Shortly after, Tulyakov \cite{tulyakov2009system} derived a system of related recurrences that recast these approximations into a new structural form.

Rivoal \cite{rivoal2009rational} then developed a different approximation based on simultaneous Padé approximants, constructing a third-order Apéry-like recurrence for \(\gamma+\log(x)\), in particular recovering Euler's constant when \(x=1\). He also constructed a sixth-order recurrence
giving simultaneous rational approximations to \(\gamma\) and
\(\zeta(2)-\gamma^2\).

Concurrently, Hessami Pilehrood and Hessami Pilehrood
\cite{pilehrood2013continued} studied related constructions and gave a Meijer G-function formulation, producing a second-order inhomogeneous linear recurrence that similarly converges to \(\gamma\) or \(\delta\).

In his survey on \(\gamma\), Lagarias \cite{lagarias2013euler} emphasizes the novelty of these constructions: \emph{only these explicit holonomic recurrences converging to \(\gamma\) are known in the literature}.

In this work, we propose and demonstrate a systematic approach for generating holonomic sequences that converge to $\gamma$, enabled by a unique connection between the Meijer G-function and $\gamma$.

By applying the contiguous relations of the Meijer G-Function we construct a \emph{Conservative Matrix Field} (CMF) \cite{elimelech2024algorithm, weinbaum2025conservative} of rank $r \in \mathbb{Z}_{\ge3}$ and dimension $d=2r-1$.
The CMF serves as a generator with an infinite parameter space, where every choice of source and trajectory $\mathbf{s}, \mathbf{t} \in \mathbb{Z}^{d}$ yields a \emph {Trajectory Matrix} $\mathcal{M}_{\mathbf{s}, \mathbf{t}}(n) \in \text{GL}_r(\mathbb{Q}(n))$ that represents a linear difference system.

We derive two initial value vectors $\mathbf{p}_r, \mathbf{q}_r\in \mathbb{Z}^r$ using a novel identity between the Meijer G-function and $\gamma$, and construct an infinite family of trajectories $\mathcal{T}^*_3 \subset \mathbb{Z}^{5}$ such that for any trajectory $\mathbf{t} \in \mathcal{T}^*_3$ and for any generic vector $\mathbf{v} \in \mathbb{R}^r$, the dual limit of the CMF (\cref{sec:background-and-notation}) satisfies:
\[
\mathcal{D}_{\mathbf{0}, \mathbf{t}}^{\mathcal{M}}(\mathbf{p}_r,\mathbf{q}_r,\mathbf{v}) = \lim_{n\to\infty} \frac{\mathbf{p}_r^T \cdot \Phi_{\mathcal{M}_{\mathbf{0}, \mathbf{t}}}(n)^{-T}\cdot \mathbf{v}}{\mathbf{q}_r^T\cdot\Phi_{\mathcal{M}_{\mathbf{0}, \mathbf{t}}}(n)^{-T}\cdot \mathbf{v}} = \gamma,
\]
where the transition matrix of $M(n) \in  \text{GL}_r(\mathbb{Q}(n))$ is denoted by $\Phi_M(n):=M(0) \cdot \ldots \cdot M(n-1)$. 

We additionally propose a broader family of trajectories $\mathcal{T}_r$ for arbitrary rank $r\geq3$, containing $\mathcal{T}_3^*$ when $r=3$, whose convergence is conjectured and supported by numerical evidence.

Our evaluation at the origin is achieved through the composition of two mechanisms, operating on the state vector of the CMF, denoted $\mathcal{S}_r(\mathbf{s})$, at the origin:
\begin{enumerate}
\item \textbf{The Algebraic Component} (\cref{sec:g-function-identity}): We derive two initial value vectors $\mathbf{p}_r, \mathbf{q}_r \in\mathbb{Z}^r$ such that
\[
\frac{\mathbf{p}_r^T\mathcal{S}_r(\mathbf{0})}{\mathbf{q}_r^T\mathcal{S}_r(\mathbf{0})}=\gamma.
\]

\item \textbf{The Dynamical Component} (\cref{sec:dual-limit-convergence}): For any trajectory $\mathbf{t} \in \mathcal{T}^*_3$ and generic vector $\mathbf{v} \in \mathbb{R}^r$, the dual system of the trajectory matrix isolates the initial state vector $\mathcal{S}_r(\mathbf{0})$ up to scaling. More precisely, there exists a sequence $c_n\in\mathbb{R}\setminus\{0\}$ such that
\[
\lim_{n\to\infty}
c_n\Phi_{\mathcal{M}_{\mathbf{0}, \mathbf{t}}}(n)^{-T}\mathbf{v}
=
\mathcal{S}_r(\mathbf{0}).
\]
\end{enumerate}

For example (derived in \cref{sec:examples}), the linear difference system $\mathbf{x}_{n+1}^T=\mathbf{x}_n^T M(n)$ defined by
\[
M(n) = \begin{bmatrix} 
    1 & n^{3} & (n-1)(3n+3)(n^2+n+1) \\
    2 & 1 - 3 n^{2} & - \left(n + 1\right) \left(8 n^{2} + n + 5\right)\\
    1 & 3 n + 2 & 6 n \left(n + 1\right)
\end{bmatrix}
\]

and initialized by a generic seed (such as $\mathbf{v}^T = \begin{pmatrix} 0 & 0 & 1 \end{pmatrix}$) satisfies
\[
\lim_{n\to\infty} \begin{bmatrix} 1 & 1 & 0 \\ 1 & 1 & 1 \end{bmatrix} \cdot \Phi_M(n)^{-T} \cdot \mathbf{v} \propto \begin{bmatrix} \gamma \\ 1 \end{bmatrix}.
\]

Each linear difference system can then be converted into an Ap\'ery-like recurrence of order $r$ using the cyclic vector method \cite{bostan2013complexity}, for which a gauged set of initial values yields the same approximation.

Even though empirical analysis indicates that the convergence of our approximations is insufficient to yield an irrationality proof for $\gamma$, we hope that broadening the space of such sequences and analyzing their matrix properties may ultimately contribute to a proof of the irrationality of $\gamma$.

The construction presented in this paper is summarized in a four-stage symbolic algorithm (\cref{fig:algorithm-pipeline}), which transforms analytic contiguous relations into discrete dynamical systems.
\begin{enumerate}
  \item Construct a CMF for the Meijer G-function of rank $r \ge 3$ (\cref{sec:meijer-g-cmf}).
  \item Derive the initial values $\mathbf{p}_r, \mathbf{q}_r \in \mathbb{Z}^r$ (\cref{sec:g-function-identity}).
  \item Choose a trajectory \(\mathbf{t} \in \mathcal{T}_r\) (\cref{sec:dual-limit-convergence}).
  \item Convert \(\mathcal{M}_{\mathbf{s}, \mathbf{t}}(n)^{-T}\) into a holonomic sequence using the cyclic vector method (\cref{def:cyclic-vector-method}).
\end{enumerate}

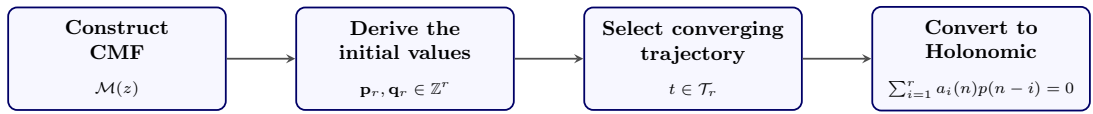
\begin{figure}[H]
  \centering
  \begin{tikzpicture}[
    % Define a global style for the boxes
    box/.style={
        draw=blue!40!black,       % Elegant dark slate-blue border
        fill=blue!3,              % Very subtle light blue fill
        rounded corners=4pt,      % Slightly sharper, more modern corners
        thick,
        align=center,             % Automatically centers multi-line text
        minimum width=3.4cm,      % Uniform width
        minimum height=1.6cm,     % Uniform height
        font=\small\bfseries,
        inner sep=2mm
    },
    % Define a global style for the arrows using universally supported stealth
    >=stealth,
    arrow/.style={
        ->,                       % Standard directed arrow
        thick,
        draw=black!70             % Slightly softened black for elegance
    }
]

  % Place nodes at absolute coordinates to prevent library conflicts
  \node[box] (b1) at (0, 0) {Construct \\ CMF \\[1.5mm] \normalfont\scriptsize $\mathcal{M}(z)$};
  \node[box] (b2) at (4.5, 0) {Derive the \\ initial values \\[1.5mm] \normalfont\scriptsize $\mathbf{p}_r, \mathbf{q}_r \in \mathbb{Z}^r$};
  \node[box] (b3) at (9.0, 0) {Select converging \\ trajectory \\[1.5mm] \normalfont\scriptsize $t \in \mathcal{T}_r$};
  \node[box] (b4) at (13.5, 0) {Convert to \\ Holonomic \\[1.5mm] \normalfont\scriptsize $\sum_{i=1}^{r} a_i(n)p(n-i)=0$};

  % Draw the arrows connecting the nodes dynamically
  \draw[arrow] (b1) -- (b2);
  \draw[arrow] (b2) -- (b3);
  \draw[arrow] (b3) -- (b4);

\end{tikzpicture}
  \caption{The symbolic algorithm for generating holonomic sequences converging to $\gamma$.}
  \label{fig:algorithm-pipeline}
\end{figure}

\section{Background and Notation}
\label{sec:background-and-notation}

\subsection{Special Functions}

\begin{definition}
Euler's constant, also known as the Euler-Mascheroni constant, is defined as~\cite{lagarias2013euler}:
$$ \gamma = -\int_{0}^{\infty} e^{-x} \ln x \, dx = 0.57721... $$
\end{definition}

\begin{definition}[Gamma Function]
\label{def:gamma-function}
The gamma function \(\Gamma(z)\) is defined for \(\operatorname{Re}(z)>0\) by
\[
\Gamma(z)=\int_0^\infty t^{z-1}e^{-t}\,dt.
\]
It extends meromorphically to the complex plane, satisfying
\( \Gamma(z+1)=z\Gamma(z)\),
and more specifically \( \Gamma(n)=(n-1)! \) for any \(n\in\mathbb{N}\).
It has simple poles at the non-positive integers.
\end{definition}

\begin{definition}[Digamma and Trigamma Functions]
\label{def:digamma-function}
The digamma function is the logarithmic derivative of the gamma function:
\[
\psi(z)=\frac{d}{dz}\log\Gamma(z)=\frac{\Gamma'(z)}{\Gamma(z)}.
\]
Its derivative \( \psi_1(z)=\psi'(z) \) is called the trigamma function.
\end{definition}

\begin{lemma}[Gamma Asymptotics]
\label{lem:gamma-asymptotics}
Uniformly as \(|z|\to\infty\) in any closed sector avoiding the negative real
axis, the principal branch of the logarithm satisfies
\[
\log\Gamma(z)
=
\left(z-\frac12\right) \log z-z+\frac12\log(2\pi)
+\mathcal{O}(z^{-1}).
\]
Consequently,
\[
\log\Gamma(z)=z\log z-z+\mathcal{O}(\log |z|),
\]
\[
\psi(z)=\log z+\mathcal{O}(z^{-1}),
\]
and
\[
\psi_1(z)=z^{-1}+\mathcal{O}(z^{-2}).
\]
Moreover, for fixed \(\sigma\) in a bounded interval,
\[
|\Gamma(\sigma+it)|
=
\mathcal{O}\left(
|t|^{\sigma-\frac12}e^{-\frac{\pi}{2}|t|}
\right)
\qquad (|t|\to\infty).
\]
\end{lemma}

\begin{proof}
These are the standard Stirling, digamma, and polygamma asymptotic estimates;
see \cites[Section 5.11(i)]{NIST:DLMF}[Section 5.15]{NIST:DLMF}.
\end{proof}

\begin{definition}[D-finite function]
Let \(\mathbb{K}\) be a differential field.
A function \(f(x)\) is called \emph{D-finite} (or \emph{holonomic}) over \(\mathbb{K}\) if there exist polynomials \(p_0(x),\dots,p_r(x)\in\mathbb{K}[x]\), not all zero, such that
\[
p_0(x)\,f(x) \;+\; p_1(x)\,f'(x) \;+\; \cdots \;+\; p_r(x)\,f^{(r)}(x) \;=\; 0,
\]
for some \(r\in\mathbb{N}_{0}\). The minimal such \(r\) is the \emph{order} of \(f\).
\end{definition}

\begin{definition}[Meijer G-function]
\label{def:meijer-g}
Let $m, n, p, q \in \mathbb{N}_0$ such that $0 \leq n \leq p$ and $0 \leq m \leq q$.

The Meijer G-function~\parencites[§16.17, Eq.~(16.17.1)]{NIST:DLMF}[\S5.8]{LukeYudellL1969Tsfa}
 is a special function defined by the following contour integral:
$${G}_{p,q}^{m,n} \left( \begin{matrix} a_1, \dots, a_p \\ b_1, \dots, b_q \end{matrix} \,\middle|\, z \right) = \frac{1}{2\pi i} \int_L \frac{ \prod_{j=1}^m \Gamma(b_j - s) \prod_{j=1}^n \Gamma(1 - a_j + s) }{ \prod_{j=m+1}^q \Gamma(1 - b_j + s) \prod_{j=n+1}^p \Gamma(a_j - s) } z^s \, ds.$$
Where $\Gamma(n)$ is the gamma function.
The choice of contour $L$ depends on the specific parameters of the G-function (see Fig. \ref{fig:meijerG-contours}, adapted from the \emph{NIST Digital Library of Mathematical Functions} (DLMF), Fig.~\href{https://dlmf.nist.gov/16.17.F1}{16.17.1}, Release 1.2.4 (2025-03-15)).
There are three options for the contour, and the definition dictates to choose the contour for which it converges (if it converges for more than one, then the different options agree):

\begin{enumerate}
    \item \(L\) goes from \(-i\cdot\infty\) to \(i\cdot\infty\); the integral converges if \(p+q<2(m+n)\) and \(\lvert\arg z\rvert<(m+n-\tfrac12(p+q))\pi\).
    \item \(L\) is a loop source at infinity on a line parallel to the positive real axis, encircling the poles of \(\Gamma(b_\ell-s)\) once in the negative sense and returning to infinity on another line parallel to the positive real axis; the integral converges for all \(z\neq0\) if \(p<q\), and for \(0<\lvert z\rvert<1\) if \(p=q\ge1\).
    \item \(L\) is a loop source at infinity on a line parallel to the negative real axis, encircling the poles of \(\Gamma(1-a_j+s)\) once in the positive sense and returning to infinity on another line parallel to the negative real axis; the integral converges for all \(z\) if \(p>q\), and for \(\lvert z\rvert>1\) if \(p=q\ge1\).
\end{enumerate}

\begin{figure}[h]
  \centering
  \begin{subfigure}[t]{0.32\textwidth}
    \centering
    \includegraphics[width=\linewidth]{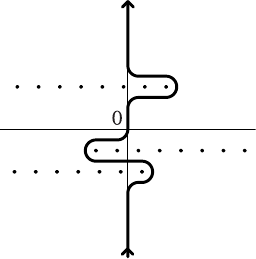}
    \caption{Case (i)}
  \end{subfigure}\hfill
  \begin{subfigure}[t]{0.32\textwidth}
    \centering
    \includegraphics[width=\linewidth]{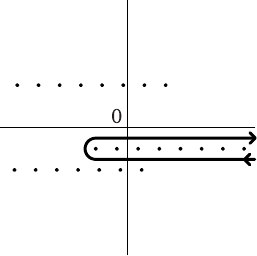}
    \caption{Case (ii)}
  \end{subfigure}\hfill
  \begin{subfigure}[t]{0.32\textwidth}
    \centering
    \includegraphics[width=\linewidth]{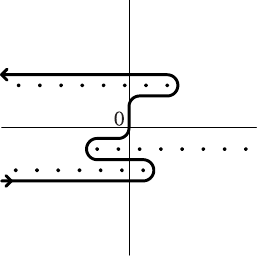}
    \caption{Case (iii)}
  \end{subfigure}
  \caption{\footnotesize
    Contours $L$ in the $s$-plane for the integral representation of the Meijer $G$-function. 
    (i) A vertical contour separating the poles of $\Gamma(b_j - s)$ and $\Gamma(1 - a_j + s)$, used when $p+q<2(m+n)$. 
    (ii) A contour looping around the poles of $\Gamma(b_j - s)$ (negative orientation), used when $p\le q$. 
    (iii) A contour looping around the poles of $\Gamma(1 - a_j + s)$ (positive orientation), used when $p\ge q$. 
    In each case, the contour is chosen so that the integral converges, and different valid choices yield the same function. 
    Adapted from the NIST Digital Library of Mathematical Functions, Fig.~16.17.1.
    }
  \label{fig:meijerG-contours}
\end{figure}

The Meijer G-function is D-finite. It satisfies the following differential equation:
{\small
\[
\left[\,(-1)^{p-m-n}\, z \prod_{j=1}^{p}\!\left( z\frac{d}{dz} - a_j + 1 \right)
\;-\; \prod_{j=1}^{q}\!\left( z\frac{d}{dz} - b_j \right)\right] G^{m, n}_{p, q}\left(\begin{matrix} a_1,\dots,a_p \\ b_1,\dots, b_q\end{matrix};z\right) = 0.
\]

}
\end{definition}

\subsection{Difference Equations and Holonomic Sequences}

\begin{definition}[Linear Difference System]
\label{def:linear-difference-system}
Let $r \in \mathbb{N}$. A homogeneous first-order linear difference system is defined by the recurrence relation:
\[
    \mathbf{x}({n+1})^T = \mathbf{x}(n)^T M(n), \quad n \in \mathbb{Z}
\]
where $\mathbf{x}(n) : \mathbb{Z} \to \mathbb{C}^r$ is the solution at step $n$, and $M(n): \mathbb{Z} \to \mathrm{GL}_r(\mathbb{Q}(n))$ is a rational invertible matrix \footnote{Most sources define systems using column vectors and left-sided multiplication: $\mathbf{x}(n+1)=M(n)\mathbf{x}(n)$. These definitions are equivalent up to a transpose operation}.
\end{definition}

\begin{definition}[Transition Matrix]
\label{def:fundamental-matrix}
The transition matrix is defined as the following product:
$$\Phi_M(n) := M(0) \cdot M(1) \cdot \dotsc \cdot M(n-1)$$
The solution of the linear difference system at step $n$ is then given as $\mathbf{x}(n)^T = \mathbf{x}(0)^T \cdot\Phi_M(n) $, where $\mathbf{x}(0)$ is the initial values vector at step 0.
\end{definition}

\begin{definition}
The dual system of a linear difference system defined by $M(n)$ is the linear difference system defined by $M(n)^{-T}$. Its transition matrix satisfies $$\Phi_{M^{-T}}(n)=M(0)^{-T} \cdot\dotsc\cdot M(n-1)^{-T}=\big(M(0) \cdot\dotsc\cdot M(n-1)\big)^{-T}=\Phi_M(n)^{-T}.$$
\end{definition}

\begin{definition}[Holonomic Sequences and Ap\'ery-like Recurrences]
Let $\mathbb{K}$ be a field of characteristic zero. A sequence $p: \mathbb{Z} \to \mathbb{Q}$ is said to be \textbf{holonomic} of order $r$ if there exist polynomials $a_0(n), \dots, a_r(n) \in \mathbb{K}[n]$ with $a_0(n) \neq 0$ and $a_r(n) \neq 0$ such that for any $n\geq0$:
\[
    \sum_{i=0}^{r} a_i(n) p(n-i) = 0.
\]

The recurrence it satisfies is called an Ap\'ery-like recurrence (or a P-recursive equation).
\end{definition}

\begin{remark}
\label{rem:companion-form}
An Ap\'ery-like recurrence $\sum_{i=0}^{r}a_{i}(n)p(n-i) = 0$ can be represented using a companion matrix form:

$$
C(n)=\begin{bmatrix}
    0 & 0 & \hdots& 0 & 0 & -\frac{a_{r}(n)}{a_0(n)} \\
    1 & 0 & \hdots & 0 & 0 & -\frac{a_{r-1}(n)}{a_0(n)} \\
    0 & 1 & \hdots & 0 & 0 & -\frac{a_{r-2}(n)}{a_0(n)} \\
    \vdots & \vdots & & \vdots & \vdots &\vdots \\
    0 & 0 & \hdots & 1 & 0 & -\frac{a_2(n)}{a_0(n)} \\
    0 & 0 & \hdots & 0 & 1 & -\frac{a_1(n)}{a_0(n)}
\end{bmatrix}
$$

With a transition matrix of the form:
$$
\Phi_C(n) =
\begin{bmatrix}
    p_{1}({n - r}) & p_{1}({n + 1 - r}) & \hdots & p_{1}({n-1}) \\
    p_{2}({n - r}) & p_{2}({n + 1 - r}) & \hdots & p_{2}({n-1}) \\
    \vdots & \vdots & & \vdots & \\
    p_{r}({n - r}) & p_{r}({n + 1 - r}) & \hdots & p_{r}({n-1}) \\
\end{bmatrix}$$

\end{remark}

\begin{corollary}[Reduction to First-Order System]
\label{cor:holonomic-is-equation-system}
Every Ap\'ery-like recurrence of order $r$ can be interpreted as a first-order matrix system of dimension $r$ defined by its companion form matrix.
\end{corollary}

\begin{definition}[Cyclic Vector Method]
\label{def:cyclic-vector-method}
Given a matrix $M(n) \in \mathrm{GL}_r(\mathbb{Q}(n))$ defining a
\emph{linear difference system}, the \emph{cyclic vector method} uncoupling
algorithm \cite[Theorem B.2]{hendricks1999solving} finds a matrix
$U(n) \in \mathrm{GL}_r(\mathbb{Q}(n))$ such that
\[
C(n) \coloneq U(n)^{-1} M(n) U(n+1)
\]
is a companion form matrix representing an Ap\'ery-like recurrence, where
the rightmost column is made of rational functions in $n$.

If $U(n)$ is regular and invertible at every $n\in\mathbb{N}_0$, the two
systems are gauge-equivalent along the entire forward orbit, and their
transition matrices satisfy
\[
\Phi_C(n)=U(0)^{-1}\Phi_M(n)U(n).
\]
\end{definition}

\begin{definition}[Formal Solutions and Asymptotic Ordering]
\label{def:asymptotics}
Let $r \in \mathbb{N}$ and $M \in \mathrm{GL}_r(\mathbb{Q}(n))$.
By the classical asymptotic theory of linear difference equations
\cites{birkhoff1930formal}{birkhoff1933analytic}[\S2.4]{kauers2023dfinite}, the system
\[
\mathbf{x}(n+1)^T=\mathbf{x}(n)^T M(n)
\]
admits $r$ linearly independent formal generalized-series solutions
$\mathbf{u}_1(n),\dots,\mathbf{u}_r(n)$ of Birkhoff-Trjitzinsky type, with
\[
\mathbf{u}_i(n)
\sim
n!^{\kappa_i}\lambda_i^n
\exp\left(
\sum_{k=1}^{\rho_i-1}
\mu_{i,k} n^{k/\rho_i}
\right)
n^{\alpha_i}
\sum_{j=0}^{m_i}(\log n)^j
\sum_{k=0}^{\infty}
\mathbf{c}_{i,j,k}n^{-k/\rho_i}.
\]

We order these formal solutions by their asymptotic magnitude so that
\[
\|\mathbf{u}_{i+1}(n)\|
=
\mathcal{O}\big(\|\mathbf{u}_i(n)\|\big)
\qquad
\text{for all } i\in\{1,\dots,r-1\}.
\]

Moreover, there exists a basis of linearly independent solutions $\mathbf{x}_1(n),\dots,\mathbf{x}_r(n)$ satisfying
\[
\mathbf{x}_i(n)\sim\mathbf{u}_i(n), \qquad i\in\{1,\dots,r\}.
\]

The solution $\mathbf{x}_1(n)$ is termed a \emph{dominant solution}, while
$\mathbf{x}_r(n)$ is termed a \emph{recessive solution}.
\end{definition}

\begin{remark}[Asymptotics of Ap\'ery-like Recurrences]
\label{rem:recurrence-asymptotics}
Since every Ap\'ery-like recurrence of order $r$ corresponds to a companion
system of dimension $r$ (\Cref{cor:holonomic-is-equation-system}), the same
asymptotic theory applies directly to scalar recurrences.

The algorithmic construction of these expansions - determining the precise values of the exponential growth rates, algebraic factors, and logarithmic terms - is well-established.
We refer to the foundational work by Wimp and Zeilberger \cite{wimp1985resurrecting}, the \texttt{Asymptotics} package developed by the RISC group \cite{kauers2011mathematica}, and our \texttt{ramanujantools} package \cite{https://doi.org/10.5281/zenodo.15747099} for the effective computation of these expansions in the scalar case.
\end{remark}

\begin{definition}[Generic Vector]
\label{def:generic-vector}
Let $\mathbf{x}(n+1)^T = \mathbf{x}(n)^T M(n)$ be a linear difference system of dimension $r$, with $\mathbf{x}_1(n), \dots, \mathbf{x}_r(n)$ a basis of sequence solutions ordered by asymptotic magnitude, and suppose the first $k$ solutions span its dominant asymptotic subspace.

Every $\mathbf{v}\in\mathbb{C}^r$ admits a unique decomposition
\(
\mathbf{v} = \sum_{i=1}^r c_i\mathbf{x}_i(0).
\)

We say that $\mathbf{v}$ is \emph{generic} to the system if
\(
(c_1,\dots,c_k)\neq \mathbf{0}.
\)
\end{definition}

\begin{definition}[Unique Recessive Asymptotic]
\label{def:unique-recessive}
Let $r \in \mathbb{N}, M \in  \text{GL}_r(\mathbb{Q}(n))$ and consider the linear difference system:
$$\mathbf{x}(n+1)^T = \mathbf{x}(n)^T M(n)$$
Denote the formal solutions of the system as $\mathbf{u}_1(n), \dots, \mathbf{u}_r(n)$, ordered by their asymptotic growth such that for all $i \in \{1, \dots, r-1\}$:
\[
\|\mathbf{u}_{i+1}(n)\| = \mathcal{O}\big(\|\mathbf{u}_i(n)\|\big) \quad \text{as } n \to \infty.
\]

The system admits a \textbf{unique recessive asymptotic} if the last basis solution is strictly dominated by all others:
\[
\lim_{n\to\infty} \frac{\|\mathbf{u}_r(n)\|}{\|\mathbf{u}_i(n)\|} = 0, \quad \forall i \in \{1, \dots, r-1\}.
\]
\end{definition}

\begin{lemma}[Dual Convergence Criteria]
\label{lem:convergence-criteria}
Let $r \in \mathbb{N}, M(n) \in \text{GL}_r(\mathbb{Q}(n))$.
Assume the primal system admits a unique recessive asymptotic \(\mathbf u_r(n)\), and let \(\mathbf x_r(n)\) be a corresponding sequence solution satisfying
\[
\mathbf x_r(n)\sim\mathbf u_r(n).
\] 
Let $\Phi_M(n) = M(0) \cdot \dotsc \cdot M(n-1)$ be the transition matrix. Then, for any vector $\mathbf{v}\in\mathbb{C}^r$ generic to the dual system, the forward iteration of the dual system isolates the initial state $\mathbf{x}_r(0)$ up to scaling. That is, there exists a sequence $c_n\in\mathbb{C}\setminus\{0\}$ such that

\[
\lim_{n\to\infty}
c_n\Phi_M(n)^{-T}\mathbf{v}
=
\mathbf{x}_{r}(0).
\]

\end{lemma}

\begin{proof}
This convergence behavior is the geometric mechanism underlying the stability of backward recurrence algorithms, such as Miller's algorithm \cite{gautschi1967computational}.
As established in the general analysis of these algorithms, the asymptotic hierarchy of the formal solutions inverts in the dual difference system. If a primal difference system admits a unique recessive asymptotic, this minimal solution manifests as the dominant solution of the dual system.

The explicit proof that the transition matrix of the dual system isolates the initial values of the unique recessive solution is derived in \cite[Section 2]{zahar1976mathematical}.
\end{proof}

\subsection{The Conservative Matrix Field}
\label{subsec:the-conservative-matrix-field}

\begin{definition}[Conservative Matrix Field]
\label{def:CMF}
Let $d,r\in\mathbb N$. A \emph{Conservative Matrix Field} (CMF) of
rank $r$ and dimension $d$ is defined by rational basis matrices
\[
\mathcal{M}_{\mathbf e_1}(\mathbf p),\ldots,
\mathcal{M}_{\mathbf e_d}(\mathbf p)
\in
\mathrm{GL}_r\!\bigl(\mathbb Q(\mathbf p)\bigr),
\qquad
\mathbf p\in\mathbb Z^d,
\]
satisfying the conservation property for any lattice point $\mathbf{p} \in \mathbb{Z}^d$ in which they are defined:
\[
\mathcal{M}_{\mathbf e_i}(\mathbf p)\,
\mathcal{M}_{\mathbf e_j}(\mathbf p+\mathbf e_i)
=
\mathcal{M}_{\mathbf e_j}(\mathbf p)\,
\mathcal{M}_{\mathbf e_i}(\mathbf p+\mathbf e_j),
\qquad
1\le i,j\le d.
\]
\end{definition}

\begin{remark}
CMFs are a discrete equivalent of conservative vector fields. The formal definition of CMFs is a discrete equivalent of potential, while the basis matrices $\mathcal{M}_1, \dots, \mathcal{M}_d$ are the discrete equivalent of the gradient, often more convenient for algorithmic use.
\end{remark}
\begin{remark}\label{rem:CMF-negative-matrices}
We denote the negative basis matrices as $\mathcal{M}_{-\mathbf{e}_i}(\mathbf{p}) \in \text{GL}_r(\mathbb{Q}(\mathbf{p}))$, and they satisfy
$$\mathcal{M}_{-\mathbf{e}_i}(\mathbf{p})\cdot \mathcal{M}_{\mathbf{e}_i}(\mathbf{p} - \mathbf{e}_i) = I, \quad \mathcal{M}_{\mathbf{e}_i}(\mathbf{p}) \cdot \mathcal{M}_{-\mathbf{e}_i}(\mathbf{p} + \mathbf{e}_i) = I$$
therefore, $$\mathcal{M}_{-\mathbf{e}_i}(\mathbf{p}) = \mathcal{M}_{\mathbf{e}_i}(\mathbf{p} - \mathbf{e}_i)^{-1}, \quad \mathcal{M}_{\mathbf{e}_i}(\mathbf{p}) = \mathcal{M}_{-\mathbf{e}_i}(\mathbf{p}+\mathbf{e}_i)^{-1}$$

Additionally, $M_\mathbf{0}(\mathbf{p}) = I_r$ vacuously.
\end{remark}

\begin{definition}[Trajectory Matrix]
\label{def:trajectory-matrix}
Let $\mathcal{M}$ be a CMF and let $s, t\in\mathbb{Z}^d$ (representing a source point and a trajectory). A \emph{Trajectory Matrix} $\mathcal{M}_{\mathbf{s}, \mathbf{t}}(n)  \in \mathrm{GL}_r(\mathbb{Q}(n))$ is defined as the (only) matrix satisfying $\mathcal{M}_{\mathbf{s}+n \cdot \mathbf{t}} \cdot \mathcal{M}_{\mathbf{s}, \mathbf{t}}(n) = \mathcal{M}_{\mathbf{s}+(n+1) \cdot \mathbf{t}}$
\end{definition}

\begin{tcolorbox}[
  title={Calculating \(\mathcal{M}_{\mathbf{s}, \mathbf{t}}(n)\)},
  colback=blue!4, colframe=blue!40!black
]
\textbf{Input:}
\begin{enumerate}\setlength{\itemsep}{0pt}
  \item CMF basis matrices $\{\mathcal{M}_{\mathbf{e}_i}\}_{i=1}^{d}$
  \item source, trajectory $\mathbf{s}, \mathbf{t} \in \mathbb{Z}^d$
\end{enumerate}
\textbf{Output:} The trajectory matrix $\mathcal{M}_{\mathbf{s}, \mathbf{t}}(n)$

\begin{algorithmic}
\State $\mathbf{p} \gets \mathbf{s} + n\mathbf{t}$,\quad $M \gets I_r$
\For{$i = 1$ \textbf{to} $d$}
\State $\sigma \gets \operatorname{sgn}(t_i)$
  \For{$j = 1$ \textbf{to} $|t_i|$}
    \State $M \gets M \cdot \mathcal{M}_{\sigma \mathbf{e}_i}(\mathbf{p})$
    \State $\mathbf{p} \gets \mathbf{p} + \sigma \mathbf{e}_i$
  \EndFor
\EndFor
\State \Return $M$
\end{algorithmic}
\end{tcolorbox}

\begin{remark}
A trajectory matrix defines a homogeneous linear difference system of order 1 (\cref{def:linear-difference-system})
\end{remark}

\begin{remark}
A trajectory matrix can be thought of as a one-dimensoinal CMF. Its lattice is embedded into the $d$-dimensional lattice, with its origin set at $\mathbf{s}$, and every step translates into a step in our trajectory $t \in \mathbb{Z}^d$.
\end{remark}

\begin{definition}[Primal and Dual CMF limits]
\label{def:cmf-primal-dual-limits}
Let $\mathcal{M}$ be a CMF, and let $r, d \in \mathbb{N}$.

Let $\mathbf{s}, \mathbf{t} \in \mathbb{Z}^d$ and let $\mathbf{p}, \mathbf{q}, \mathbf{v} \in \mathbb{Q}^r$.
\begin{itemize}
\item The \emph{primal CMF limit}, if exists, is the ratio 
\[
\mathcal{P}_{\mathbf{s}, \mathbf{t}}^{\mathcal{M}}(\mathbf{p},\mathbf{q},\mathbf{x}) :=
\lim_{n\to\infty} \frac{\mathbf{p}^T \cdot \Phi_{\mathcal{M}_{\mathbf{s}, \mathbf{t}}}(n)\cdot \mathbf{v}}{\mathbf{q}^T\cdot\Phi_{\mathcal{M}_{\mathbf{s}, \mathbf{t}}}(n)\cdot \mathbf{v}}
\]

\item The \emph{dual CMF limit}, if exists, is the ratio
\[
\mathcal{D}_{\mathbf{s}, \mathbf{t}}^{\mathcal{M}}(\mathbf{p}, \mathbf{q}, \mathbf{v}):=
\lim_{n\to\infty} \frac{\mathbf{p}^T \cdot \Phi_{\mathcal{M}_{\mathbf{s}, \mathbf{t}}}(n)^{-T}\cdot \mathbf{v}}{\mathbf{q}^T\cdot\Phi_{\mathcal{M}_{\mathbf{s}, \mathbf{t}}}(n)^{-T}\cdot \mathbf{v}}
\]
\end{itemize}
\end{definition}

\subsection{D-finite Conservative Matrix Fields}
\label{subsec:d-finite-cmf}

A particularly common choice of CMF is a \emph{D-finite} CMF~\cite{weinbaum2025conservative}. Many mathematical constants such as $e$, $\pi$, and $\gamma$ are deeply connected to linear differential equations, and thus arise naturally as limit values of sequences generated by such CMFs.

\begin{definition}[State vector of a D-finite function]
\label{def:state-vector}
Let \( F(\mathbf{v}; z) \) be a D-finite function, and let \( \theta = z\cdot \frac{d}{dz} \) denote the Euler operator. Then the \emph{state vector} associated with \( F \) is defined as
\[
\mathcal{S}(\mathbf{v}; z) :=
\begin{pmatrix}
F(\mathbf{v}; z), &
\theta F(\mathbf{v}; z), &
\cdots &
\theta^{r-1} F(\mathbf{v}; z)
\end{pmatrix}^T
\]
where \( r \) is the order of the differential equation satisfied by \( F \), \( \mathbf{p} \in \mathbb{Z}^d \) encodes the discrete parameters of \( F \), and \( z \in \mathbb{R} \) its continuous variable.
\end{definition}

Assume further that the contiguous relations of \(F\) express each discrete parameter shift as a rational linear combination of \(F, \theta F, \dots,\theta^{r-1}F\).
Then there exists a finite set of matrices
\[
\left\{
\mathcal{M}_{\mathbf{e}_i}^{(z)}(\mathbf{v})
\right\}_{i=1}^d
\subset
\mathrm{GL}_r(\mathbb{R}(\mathbf{v}))
\]
such that for all regular \(z \in \mathbb{R}\), \( \mathbf{p} \in \mathbb{Z}^d \) and \( i \in \{1,\dots,d\} \):
\[
\mathcal{S}(\mathbf{p} + \mathbf{e}_i; z)^T = \mathcal{S}(\mathbf{p}; z)^T \cdot \mathcal{M}_{\mathbf{e}_i}^{(z)}(\mathbf{p}).
\]
For each regular \(z\in\mathbb{R}\), the set
\(\{ \mathcal{M}_{\mathbf{e}_i}^{(z)}(\mathbf{v}) \}_{i=1}^d\)
forms a CMF basis.

These relations can be derived algorithmically using creative telescoping and differential elimination techniques. For further reading, including the process for constructing a CMF for the generalized hypergeometric function $_pF_q$, and an explicit example for $_2F_1$, see~\cite[§3]{weinbaum2025conservative}.

\section{Approximation Construction Framework}
\label{sec:framework-definitions}

In this section, we present the details of our framework: The parameterization choices for the CMF, the trajectories that we use and the initial values of the two solutions with ratio approximating $\gamma$.

\begin{definition}[The $\gamma$ CMF and Meijer G-function Parameterization]
\label{def:the-cmf}
Let $r \geq 3$. We define $\mathcal{M}_r$ as the D-Finite CMF of the Meijer G-function with the parameters $(m, n, p, q) = (r, r-1, r-1, r)$, evaluated at $z=1$.

The CMF is of rank $r$, its lattice is of dimension $d=2r-1$, and its state vector at $\mathbf{v}\in\mathbb{Z}^d$ is 
\[
\mathcal{S}_r(\mathbf{v}) =
\begin{pmatrix}
G^{r, r-1}_{r-1, r}(\mathbf{v}; 1), & \theta G^{r, r-1}_{r-1, r}(\mathbf{v}; 1), & \dots & \theta^{r-1} G^{r, r-1}_{r-1, r}(\mathbf{v}; 1)
\end{pmatrix}^T
\]
\end{definition}

\begin{remark}
\label{rem:identity-contour}
The contour $L$ for ${G}^{r,r-1}_{r-1,r} \left( \mathbf{v}; z \right)$ is defined as the standard contour (for the case when $q > p \geq 0$) that encircles the poles of the $b$ parameters (i.e., the poles in the right half of the plane) exactly once in the negative direction.
\end{remark}

\begin{definition}[Initial Values]
\label{def:initial_values}
Let $r \geq 3$. We define the initial value vectors $\mathbf{p}_r, \mathbf{q}_r \in \mathbb{Z}^r$, whose entries for $0 \le k \le r-1$ are given explicitly by:

\[(\mathbf{p}_r)_k = 
\begin{cases} 
\binom{r-3}{k} & 0 \le k \le r-3 \\
(-1)^{r-1} & k = r-2 \\
0 & k = r-1
\end{cases},
\qquad
(\mathbf{q}_r)_k = 
\begin{cases} 
\binom{r-2}{k} & 0 \le k \le r-2 \\
(-1)^{r-1} & k = r-1
\end{cases}\]

\end{definition}

\begin{definition}[Conjectured Trajectories]
\label{def:conjectured-trajectories}
Let $r\ge3$, and let
\( \mathbf a\in\mathbb Z_{\le0}^{\,r-1} \),
\( \mathbf b\in\mathbb N^{\,r}\).
Define the trajectory associated with $(\mathbf a,\mathbf b)$ as
\[
\mathbf t(\mathbf a,\mathbf b)
:=
(a_1,\ldots,a_{r-1};\,b_1,\ldots,b_r)^T
\in
\mathbb Z^{2r-1}.
\]

Let $k(\mathbf b)$ denote the multiplicity of $\min(\mathbf b)$ in $\mathbf b$.
The set of conjectured convergent trajectories is
\[
\mathcal T_r
=
\{\mathbf t(\mathbf a,\mathbf b): k(\mathbf b)\ge2\}.
\]
\end{definition}

\begin{definition}[Converging Trajectories]
\label{def:converging-trajectories}
We prove convergence for a subfamily of order 3.

Let $\mathbf{a} \in {\mathbb{Z}_{\le 0}^{2}}$. Denoting $\mathbf{1}=(1, 1, 1)^T$, we define the subset of converging trajectories as:
\[
\mathcal{T}^*_3=\big\{\mathbf{t}(\mathbf{a}, \mathbf{1})\big\}
\]

Because $\mathbf{1}$ contains $r=3$ vectors of equal size, $k(\mathbf{1})>2$ therefore $\mathcal{T}^*_3\subset\mathcal{T}_3$.
\end{definition}

The main result of this work is the existence of infinite approximations to $\gamma$ via linear difference systems, all convertible to an Apéry-like approximation using the cyclic vector method (\cref{def:cyclic-vector-method}):

\begin{theorem}[Infinite $\gamma$ Apéry-like Approximations]
\label{thm:main}
Let $r = 3$. For the source point $\mathbf{0}\in\mathbb{Z}^{2r-1}$, any trajectory $\mathbf{t} \in \mathcal{T}^*_{r=3}$ and any generic vector $\mathbf{v}\in\mathbb{R}^r$ (With respect to the dual system), the dual CMF limit is evaluated as:
\[
\mathcal{D}_{\mathbf{0}, \mathbf{t}}^{\mathcal{M}}(\mathbf{p}_r,\mathbf{q}_r,\mathbf{v}) = \lim_{n\to\infty} \frac{\mathbf{p}_r^T \cdot \Phi_{\mathcal{M}_{\mathbf{0}, \mathbf{t}}}(n)^{-T}\cdot \mathbf{v}}{\mathbf{q}_r^T\cdot\Phi_{\mathcal{M}_{\mathbf{0}, \mathbf{t}}}(n)^{-T}\cdot \mathbf{v}} = \gamma,
\]

Furthermore, provided that the resulting gauge is regular and invertible
for every $n\in\mathbb{N}_0$, converting the system
$\mathcal{M}_{\mathbf{0},\mathbf{t}}(n)^{-T}$ into an Apéry-like recurrence
using the cyclic vector method maintains the evaluation under the gauged
initial values.

\end{theorem}

Our evaluation is achieved through the composition of two theorems. The first is the algebraic component, the identity that shows a connection between the state vector of $\mathcal{M}_r$ at the origin $\mathcal{S}_r(\mathbf{0})$ and $\gamma$:
\begin{theorem}[$\gamma$ Linear Form]
\label{thm:linear-form}
Let $r \geq 3$, and let $\mathcal{S}_r(\mathbf{v}; z)$ be the state vector of $\mathcal{M}_r$ CMF (\cref{def:the-cmf}). Then:
\[
\frac{\mathbf{p}_r^T\mathcal{S}_r(\mathbf{0})}{\mathbf{q}_r^T\mathcal{S}_r(\mathbf{0})}=\gamma.
\]
\end{theorem}

The second is the dynamical component, which proves that infinitely many trajectories are able to isolate the state vector at the origin:
\begin{theorem}[Dual Convergence]
\label{thm:dual-convergence}
Let $r=3$. For any trajectory $\mathbf{t} \in \mathcal{T}^*_r$ and generic vector $\mathbf{v} \in \mathbb{R}^r$, the dual system of the trajectory matrix isolates the initial state vector $\mathcal{S}_r(\mathbf{0})$ up to scaling.
More precisely, there exists a sequence $c_n\in\mathbb{R}\setminus\{0\}$ such that
\[
\lim_{n\to\infty}
c_n\Phi_{\mathcal{M}_{\mathbf{0}, \mathbf{t}}}(n)^{-T}\mathbf{v}
=
\mathcal{S}_r(\mathbf{0}).
\]

Moreover, any rational gauge $U(n) \in \text{GL}_r(\mathbb{Q}(n))$ that is regular and invertible for every non negative $n$, defining the system $C(n) = U(n)^{-1} \mathcal{M}_{\mathbf{0}, \mathbf{t}}(n)^{-T} U(n+1)$ maintains the isolation.
Namely, for any vector $\mathbf{v}\in\mathbb{R}^r$ generic to $C(n)$, there exists a sequence $\widetilde{c}_n\in\mathbb{R}\setminus\{0\}$ such that
\[
\lim_{n\to\infty}
\widetilde{c}_n\Phi_C(n)\mathbf{v}
=
U(0)^{-1}\mathcal{S}_r(\mathbf{0}).
\]
\end{theorem}

Using these two theorems, the proof of the main theorem is straight-forward:

\begin{proof}[Proof of \cref{thm:main}:]

By \cref{thm:dual-convergence}, there exists a sequence $c_n\in\mathbb{R}\setminus\{0\}$ such that
\[
c_n\Phi_{\mathcal{M}_{\mathbf{0}, \mathbf{t}}}(n)^{-T}\mathbf{v}
\longrightarrow
\mathcal{S}_r(\mathbf{0}).
\]
Therefore,
\[
\lim_{n\to\infty}
\frac{\mathbf{p}_r^T \Phi_{\mathcal{M}_{\mathbf{0}, \mathbf{t}}}(n)^{-T} \mathbf{v}}
{\mathbf{q}_r^T \Phi_{\mathcal{M}_{\mathbf{0}, \mathbf{t}}}(n)^{-T} \mathbf{v}}
=
\lim_{n\to\infty}
\frac{\mathbf{p}_r^T c_n\Phi_{\mathcal{M}_{\mathbf{0}, \mathbf{t}}}(n)^{-T} \mathbf{v}}
{\mathbf{q}_r^T c_n\Phi_{\mathcal{M}_{\mathbf{0}, \mathbf{t}}}(n)^{-T} \mathbf{v}}
\overset{(\ref{thm:dual-convergence})}{=} 
\frac{\mathbf{p}_r^T\mathcal{S}_r(\mathbf{0})}
{\mathbf{q}_r^T\mathcal{S}_r(\mathbf{0})}
\overset{(\ref{thm:linear-form})}{=}
\gamma.
\]

Because \cref{thm:dual-convergence} establishes that rational gauge transformations preserve this asymptotic isolation, applying the cyclic vector method yields an Ap\'ery-like recurrence whose dual system evaluates identically to $\gamma$ when paired with the gauge-transformed initial value vectors $\mathbf{p}_r^T U(0)$ and $\mathbf{q}_r^T U(0)$.

\end{proof}

\begin{conjecture}[Conjectured $\gamma$ Apéry-like Approximations]
\label{conj:whole-family}
For any $r \geq 3$ and for any $\mathbf{t}\in \mathcal{T}_r$ (\cref{def:conjectured-trajectories}), \cref{thm:dual-convergence} and consequently \cref{thm:main} hold.
\end{conjecture}

\vspace{10pt}

\section{The Meijer G-function and Euler's Constant Identity}
\label{sec:g-function-identity}

In this section, we prove \cref{thm:linear-form}, the algebraic component of our result, using a novel identity between the Meijer G-function and Euler's constant:

\begin{proposition}[G-function and $\gamma$ identity]
\label{prop:g-function-identity}
Let $\theta = z\frac{d}{dz}$ and $p\in\mathbb{N}, p\geq2$.
Define
$$
G_p(z) :=
{G}_{p,p+1}^{p+1,p} \left( \begin{matrix} 0, \dots, 0 \\ 0, 0, \dots, 0 \end{matrix} \,\middle|\, z \right)
$$
where ${G}^{m, n}_{p, q}$ is the Meijer G-function (\cref{def:meijer-g}).

Then the numerator and denominator operators
\[
\mathcal{N}_p:=(1 + \theta)^{p-2}+(-1)^p\theta^{p-1}, \qquad
\mathcal{D}_p:=(1 + \theta)^{p-1}+(-1)^p\theta^{p},
\]
satisfy that $$
\mathcal{N}_p[G_p](1) = \gamma, \qquad
\mathcal{D}_p[G_p](1) = 1.
$$
\end{proposition}

This identity directly proves \cref{thm:linear-form}:
\begin{proof}[Proof of \cref{thm:linear-form}:]
By the construction of the Conservative Matrix Field, the $k$-th entry of the state vector $\mathcal{S}(\mathbf{0};1)$ evaluates the differential operator $\theta^k$ applied to the basis function $G_{r-1}(z)$ at $z=1$.
Consequently, applying any polynomial differential operator $P(\theta) = \sum_{k=0}^{r-1} c_k \theta^k$ is algebraically equivalent to the inner product of its coefficient vector $\mathbf{v} = [c_0, \dots, c_{r-1}]^T$ with the state vector:
\[
\mathbf{v}^T \mathcal{S}(\mathbf{0}) = P(\theta)[G_{r-1}(z)]\Big|_{z=1}.
\]
By expanding the binomial definitions of the operators $\mathcal{N}_{r-1}(\theta)$ and $\mathcal{D}_{r-1}(\theta)$, it is immediate that $\mathbf{p}_r$ and $\mathbf{q}_r$ are precisely their respective coefficient vectors, yielding $\mathbf{p}_r^T \mathcal{S}(\mathbf{0}) = \gamma$ and $\mathbf{q}_r^T \mathcal{S}(\mathbf{0}) = 1$.
\end{proof}

Using Euler's reflection formula we define:
$$\Phi_p(s) := \Gamma(-s)^{p+1}\Gamma(1+s)^p = (-\pi)^p \Gamma(-s) \csc^p(\pi s),$$
we have $$G_p(z)=\frac{1}{2\pi i} \int_L \Phi_p(s) z^s \, ds.$$

\subsection{Integral Representation and WZ Certificates}
\label{subsec:integral-wz}

Define
\[
F(p, s) = \Phi_p(s).
\]

To evaluate the sum, we form a WZ pair by introducing rational certificates $Q(p, s)$\footnote{The construction of $Q(p,s)$ as rational certificates represents a specific, manually verifiable instance of Zeilberger's method of creative telescoping \cite{zeilberger1991method}. In a broader context, algorithms due to Chyzak \cite{chyzak2000extension} generalize this approach to arbitrary holonomic functions. From this perspective, our operators $\mathcal{D}_p$ and $\mathcal{N}_p$ are naturally suited for such algorithmic treatment, and the fact that we can carry out the telescoping ‘by hand’ is merely a reflection of the simplicity of the integrand.}: 
\[
G(p, s) := Q(p, s)\,F(p, s).
\]

\begin{lemma}\label{lem:theta-insertion}
For any $i \in \mathbb{N}_0$:
\[
\theta^i G_p(z) \;=\; \frac{1}{2\pi i} \int_L s^i \,\Phi_p(s)\, z^s \, ds.
\]
\end{lemma}
\begin{proof}
This follows immediately from standard operational properties of the inverse Mellin transform, where applying the differential operator $\theta = z\frac{d}{dz}$ corresponds to multiplying the $s$-domain integrand by $s$ \cite{flajolet1995mellin}.
\end{proof}

\begin{lemma}\label{lem:wz-pair}
Let $Q(p, s)$ be either
\[
Q_{\mathrm{num}}(p, s) = (s+1)^{p-2}, \quad
Q_{\mathrm{den}}(p, s) = (s+1)^{p-1}.
\]
Then
\[
G(p, s) - G(p, s-1) = \Phi_p(s)P(p, s),
\]
where $P(p, s)$ corresponds to $Q(p,s)$ as follows:
\[
P_{\mathrm{num}}(p, s) = (s+1)^{p-2} + (-1)^p s^{p-1}, \quad 
P_{\mathrm{den}}(p, s) = (s+1)^{p-1} + (-1)^p s^p.
\]
\end{lemma}

\begin{proof}
The result is obtained by direct verification using the functional identities $\Gamma(1-s)=-s\Gamma(-s)$ and $\csc(\pi(s-1))=-\csc(\pi s)$. 
Applying these to $\Phi_p(s) = (-\pi)^p \Gamma(-s) \csc^p(\pi s)$ yields:
\[
\Phi_p(s-1) = (-1)^{p+1} s \Phi_p(s).
\]
Therefore,
\begin{align*}
G(p,s)-G(p,s-1) &= Q(p,s)\Phi_p(s) - Q(p,s-1)\Phi_p(s-1) \\
&= \Phi_p(s) \bigl(Q(p,s) - Q(p,s-1)(-1)^{p+1} s\bigr) \\
&= \Phi_p(s) \bigl(Q(p,s) + (-1)^p s \,Q(p,s-1)\bigr).
\end{align*}
Substituting $Q_{\mathrm{num}}$ and $Q_{\mathrm{den}}$ into this expression directly produces $P_{\mathrm{num}}$ and $P_{\mathrm{den}}$.
\end{proof}

\begin{lemma}
\label{lem:vanishing-boundary}
For any integer $p \ge 2$ and polynomial $Q(p,s)$ defined as in \cref{lem:wz-pair},
\[
\lim_{N \to \infty} \operatorname{Res}_{s=N} Q(p,s)\Phi_p(s) = 0.
\]
\end{lemma}
\begin{proof}
By the standard residue representation of the Meijer \(G\)-function, for the right-loop contour \(L\) and \(q=p+1>p\), the contour integral is represented by the sum of the residues at the right-hand poles.
This residue expansion converges for all \(z\neq 0\).
In particular, it converges absolutely at \(z=1\).

By \cref{lem:theta-insertion}, multiplying the integrand \(\Phi_p(s)\) by the WZ certificate polynomial \(Q(p,s)\) corresponds functionally to applying the finite-order differential operator \(Q(p,\theta)\) to \(G_p(z)\). Since term-by-term differentiation preserves the radius of convergence of a power series \cite{ahlfors1979complex}, the residue sum for the modified integrand \(Q(p,s)\Phi_p(s)\) remains absolutely convergent at \(z=1\).

By the necessary condition for convergence of an infinite series, its individual terms tend to zero. Therefore,
\[
\lim_{N\to\infty}
\operatorname{Res}_{s=N} Q(p,s)\Phi_p(s)=0.
\]
\end{proof}

\begin{lemma}\label{lem:residues-values}
\[
\operatorname{Res}_{s=-1} Q_{\mathrm{num}}(p, s)\Phi_p(s) = \gamma,
\qquad
\operatorname{Res}_{s=-1} Q_{\mathrm{den}}(p, s)\Phi_p(s) = 1.
\]
\end{lemma}
\begin{proof}
Set \(t := s+1\), so \(s = -1 + t\) and \(t \to 0\) corresponds to \(s \to -1\).
Recall that
\[
\Phi_p(s) \;=\; (-\pi)^p\,\Gamma(-s)\,\csc^p(\pi s).
\]
Expanding the Laurent series of each factor at \(s=-1\) (i.e., at \(t=0\)):
\begin{enumerate}
\item Using \(\Gamma(1+u)=1-\gamma u+O(u^2)\) with \(u=-t\),
\[
\Gamma(1-t) \;=\; 1 + \gamma t + O(t^2).
\]
\item Since \(\sin(\pi s)=\sin(-\pi+\pi t)=-\sin(\pi t)=-\pi t+O(t^3)\), we get
\[
\csc(\pi s) \;=\; \frac{1}{\sin(\pi s)} \;=\; -\frac{1}{\pi t} + O(t).
\]
\end{enumerate}
Substituting \(s=-1+t\) into \(\Phi_p\), we get
\[
\Phi_p(s) 
\;=\; (-\pi)^p\,\Gamma(1-t)\,\bigl(\csc(\pi(-1+t))\bigr)^p
\;=\; (-\pi)^p\,\Gamma(1-t)\,\bigl(-\tfrac{1}{\pi t}+O(t)\bigr)^p.
\]
Using the above expansions and factoring the leading term of the \(p\)-th power,
\[
\Phi_p(s) 
\;=\; (-\pi)^p\,\Gamma(1-t)\,\Bigl[(-1)^p \pi^{-p} t^{-p}\bigl(1+O(t^2)\bigr)\Bigr]
\;=\; \underbrace{(-\pi)^p(-1)^p\pi^{-p}}_{=(-1)^{2p}=1}\,t^{-p}\,\Gamma(1-t)\,\bigl(1+O(t^2)\bigr).
\]
Hence the constants cancel exactly and
\[
\Phi_p(s) \;=\; t^{-p}\,\bigl(1+\gamma t + O(t^2)\bigr).
\]

Multiplying by the certificates, we get:
\begin{enumerate}
\item 
\(
Q_{\mathrm{num}}(p,s)\,\Phi_p(s) 
\;=\; t^{p-2}\cdot t^{-p}\,\bigl(1+\gamma t + O(t^2)\bigr)
\;=\; t^{-2} + \gamma t^{-1} + O(1),
\)
\item \(
Q_{\mathrm{den}}(p,s)\,\Phi_p(s) 
\;=\; t^{p-1}\cdot t^{-p}\,\bigl(1+\gamma t + O(t^2)\bigr)
\;=\; t^{-1} + \gamma + O(t).
\)
\end{enumerate}

Thus \[\operatorname{Res}_{s=-1} Q_{\mathrm{den}}(p,s)\Phi_p(s)=1, \qquad
\operatorname{Res}_{s=-1} Q_{\mathrm{num}}(p,s)\Phi_p(s)=\gamma.\]
\end{proof}

\subsection{Proof of~\texorpdfstring{\Cref{prop:g-function-identity}}{Proposition~\ref{thm:g-function-identity}}}
\begin{proof}

By the inverse Mellin transform property (\cref{lem:theta-insertion}), evaluating the operators $\mathcal{N}_p$ and $\mathcal{D}_p$ at $z=1$ maps the polynomials in $\theta$ directly to our polynomials $P_{\mathrm{num}}$ and $P_{\mathrm{den}}$ in $s$. Therefore:
\begin{align*}
\mathcal{N}_p[G_p](1) &= \frac{1}{2\pi i}\int_{L} P_{\mathrm{num}}(p, s)\Phi_p(s)\,ds, \\
\mathcal{D}_p[G_p](1) &= \frac{1}{2\pi i}\int_{L} P_{\mathrm{den}}(p, s)\Phi_p(s)\,ds.
\end{align*}

We evaluate these integrals by closing the contour to the right and summing the residues at $s=n \ge 0$. The clockwise closure orientation of the contour contributes an overall minus sign. By the WZ pair identity (\cref{lem:wz-pair}), the sum telescopes:
\begin{align*}
-\sum_{n=0}^{N}\operatorname{Res}_{s=n}\!\bigl(\Phi_p(s)P(p,s)\bigr) &= -\sum_{n=0}^{N} \operatorname{Res}_{s=n}\!\bigl(G(p,s)-G(p,s-1)\bigr) \\
&= \operatorname{Res}_{s=-1}G(p,s) - \operatorname{Res}_{s=N}G(p,s).
\end{align*}

Taking the limit as $N \to \infty$, the boundary term at $N$ vanishes by \cref{lem:vanishing-boundary}. We are left exactly with the residue at $s=-1$. 

Substituting the specific certificates, \cref{lem:residues-values} gives:
\begin{align*}
\mathcal{D}_p[G_p](1) &= \operatorname{Res}_{s=-1} Q_{\mathrm{den}}(p,s)\Phi_p(s) = 1, \\
\mathcal{N}_p[G_p](1) &= \operatorname{Res}_{s=-1} Q_{\mathrm{num}}(p,s)\Phi_p(s) = \gamma.
\end{align*}
\end{proof}

\begin{remark}[$\gamma$ as a single rank-three value form]
Selecting $p=2$, \cref{prop:g-function-identity} yields \((1+\theta)G_2(1)=\gamma\).

Using the Mellin--Barnes definition of the Meijer G-function \cite[§16.17, Eq.~(16.17.1)]{NIST:DLMF}, the operator
\(1+\theta\) multiplies the integrand by \(1+s\), and therefore replaces one
factor \(\Gamma(1+s)\) by \(\Gamma(2+s)\). 

Equivalently, this shifts one upper
parameter from \(0\) to \(-1\). Therefore:
\[
\gamma
=
G^{3,2}_{2,3}\!\left(
\begin{matrix}
-1,0\\
0,0,0
\end{matrix}
\,\middle|\,1
\right).
\]
\end{remark}

\subsection{Examples}
We illustrate the operators $\mathcal{N}_p$ and $\mathcal{D}_p$ for two small values of $p=r-1$, and derive $\mathbf{p}_r, \mathbf{q}_r$ from them.
The coefficients of these operators form Pascal's triangle: they consist of row $p-1$ and the $p-2$ respectively (due to the fact that $(1 + \theta)^n = \sum_{i=0}^{n}\binom{n}{i}\theta^i$), followed by $(-1)^p$. 

\begin{example}
\label{ex:identity-2}
\[
\mathcal{N}_2 = \sum_{i=0}^{0} \binom{0}{i} \theta^i + (-1)^2 \theta^1 = 1 + \theta,
\]
\[
\mathcal{D}_2 = \sum_{i=0}^{1} \binom{1}{i} \theta^i + (-1)^2 \theta^2 = 1 + \theta + \theta^2.
\]
Using the theorem, we get:
\[
\mathcal{N}_2[G_2](1) = {G_2(1) + \theta G_2(1)} = \gamma,
\]
\[
\mathcal{D}_2[G_2](1) = {G_2(1) + \theta G_2(1) + \theta^2 G_2(1)} = 1,
\]

Yielding the initial values $\mathbf{p}_3 = (1, 1, 0)^T$ and $\mathbf{q}_3 = (1, 1, 1)^T$.
\end{example}

\begin{example}
\label{ex:identity-5}
\[
\mathcal{N}_5 = \sum_{i=0}^{3} \binom{3}{i} \theta^i + (-1)^5 \theta^4 = 1 + 3\theta + 3\theta^2 + \theta^3 - \theta^4,
\]
\[
\mathcal{D}_5 = \sum_{i=0}^{4} \binom{4}{i} \theta^i + (-1)^5 \theta^5 = 1 + 4\theta + 6\theta^2 + 4\theta^3 + \theta^4 - \theta^5,
\]
Using the theorem, we get:
\[
\mathcal{N}_5[G_5](1) = {G_5(1) + 3\theta G_5(1)+ 3\theta^2 G_5(1)+ \theta^3 G_5(1) - \theta^4 G_5(1)} = \gamma,
\]
\[
\mathcal{D}_5[G_5](1) = {G_5(1) + 4\theta G_5(1) + 6\theta^2 G_5(1) + 4\theta^3 G_5(1) +\theta^4 G_5(1) -\theta^5 G_5(1)} = 1,
\]

Yielding the initial values $\mathbf{p}_6 = (1, 3, 3, 1, -1, 0)^T$ and $\mathbf{q}_6 = (1, 4, 6, 4, 1, -1)^T$.
\end{example}

\section{The Meijer G-function CMF}
\label{sec:meijer-g-cmf}
In this section, we will explain the symbolic process required to construct a D-finite CMF (\cref{subsec:d-finite-cmf}) for the Meijer G-function (\cref{def:meijer-g}).
The construction is defined by the Meijer G-function parameters $m, n, p, q \in \mathbb{N}_0$, and it is well defined for the same parameter ranges as the Meijer G-function: $0 \leq m \leq q$ and $0 \leq n \leq p.$

Denoting $r = \max(p, q)$, and $d = p + q$, we order the parameters as $\mathbf{p} = (a_1, \dots, a_p, b_1, \dots, b_q)$.

We define the state vector as:
\[
\mathcal{S}(\mathbf{p};z) :=
\begin{pmatrix}
G^{m, n}_{p, q}(\mathbf{p}; z), \theta G^{m, n}_{p, q}(\mathbf{p} ;z), \dots, \theta^{r-1}G^{m, n}_{p, q}(\mathbf{p}; z)
\end{pmatrix}^T
\]

In order to construct the CMF, we must construct identities between $\mathcal{S}(\mathbf{p}+\mathbf{e}_i)$ and $\mathcal{S}(\mathbf{p})$ that rely purely on the $\theta$ operator.

As explained in~\cref{def:meijer-g}, the Meijer G-function satisfies the following differential equation:
\[
\left[\,(-1)^{p-m-n}\, z \prod_{j=1}^{p}\!\left( z\frac{d}{dz} - a_j + 1 \right)
\;-\; \prod_{j=1}^{q}\!\left( z\frac{d}{dz} - b_j \right)\right] G^{m, n}_{p, q}\left(\begin{matrix} a_1,\dots,a_p \\ b_1,\dots, b_q\end{matrix};z\right) = 0.
\]

Similarly to the technique presented in~\cite[§3]{weinbaum2025conservative}, when looking at the basis spanned by the state vector, we denote $t_i$ as the coefficient of $\theta^i$ of the differential equation and construct
\[
M_\theta(\mathbf{p}; z) =
\begin{bmatrix}
    0 & 0 & \hdots & 0 & -\frac{t_0(\mathbf{p}; z)}{t_r(\mathbf{p}; z)} \\
    1 & 0 & \hdots & 0 & -\frac{t_1(\mathbf{p}; z)}{t_r(\mathbf{p}; z)}  \\
    0 & 1 & \hdots & 0 & -\frac{t_2(\mathbf{p}; z)}{t_r(\mathbf{p}; z)}  \\
    \vdots & \vdots & \vdots & \vdots &\vdots \\
    0 & 0 & \hdots & 1 & -\frac{t_{r-1}(\mathbf{p}; z)}{t_r(\mathbf{p}; z)} 
\end{bmatrix}
\]

Then, we use the following established relationships~\cite[\S2.4]{Mathai1993}:

\begin{align*}
\frac{d}{dz} \Biggl[ z^{1-a_i} \,
  G_{p,q}^{\,m,n} \!\left( \left. \begin{matrix} \mathbf{a_p} \\ \mathbf{b_q} \end{matrix} \; \right| z \right) \Biggr]
  &= z^{-a_i} \,
  G_{p,q}^{\,m,n} \!\left( \left. \begin{matrix} a_1, \dots, a_i-1, \dots, a_p \\ \mathbf{b_q} \end{matrix} \; \right| z \right),
  && 1 \leq i \leq n, \\[1em]
\frac{d}{dz} \Biggl[ z^{1-a_i} \,
  G_{p,q}^{\,m,n} \!\left( \left. \begin{matrix} \mathbf{a_p} \\ \mathbf{b_q} \end{matrix} \; \right| z \right) \Biggr]
  &= -z^{-a_i} \,
  G_{p,q}^{\,m,n} \!\left( \left. \begin{matrix} a_1, \dots, a_i-1, \dots, a_p \\ \mathbf{b_q} \end{matrix} \; \right| z \right),
  && n < i \leq p, \\[1em]
\frac{d}{dz} \Biggl[ z^{-b_j} \,
  G_{p,q}^{\,m,n} \!\left( \left. \begin{matrix} \mathbf{a_p} \\ \mathbf{b_q} \end{matrix} \; \right| z \right) \Biggr]
  &= -\, z^{-1-b_j} \,
  G_{p,q}^{\,m,n} \!\left( \left. \begin{matrix} \mathbf{a_p} \\ b_1, \dots, b_j + 1, \dots, b_q \end{matrix} \; \right| z \right),
  && 1 \leq j \leq m, \\[1em]
\frac{d}{dz} \Biggl[ z^{-b_j} \,
  G_{p,q}^{\,m,n} \!\left( \left. \begin{matrix} \mathbf{a_p} \\ \mathbf{b_q} \end{matrix} \; \right| z \right) \Biggr]
  &= z^{-1-b_j} \,
  G_{p,q}^{\,m,n} \!\left( \left. \begin{matrix} \mathbf{a_p} \\ b_1, \dots, b_j + 1, \dots, b_q \end{matrix} \; \right| z \right),
  && m < j \leq q, \\[1em]
\end{align*}

in order to deduce the basis matrices of the CMF:
\[
\mathcal{M}_{-\mathbf{e}_i}(\mathbf{p}; z) = 
\begin{cases}
M_\theta(\mathbf{p}; z) - (a_i - 1) I  & 1 \leq i \leq n \\[1em]
-(M_\theta(\mathbf{p}; z) - (a_i - 1) I) & n < i \leq p \\[1em]
\end{cases}
\]
and 
\[
\mathcal{M}_{e_j}(\mathbf{p}; z) =
\begin{cases}
-(M_\theta(\mathbf{p}; z) - b_j I) & p < j \leq p+m \\[1em]
M_\theta(\mathbf{p}; z) - b_j I & p+m < j \leq p+q \\[1em]
\end{cases}
\]

Finally, to finalize the CMF construction, one can invert the directions of the negative basis matrices as described in \cref{rem:CMF-negative-matrices} by calculating $\mathcal{M}_{\mathbf{e}_i}(\mathbf{p}; z) = \mathcal{M}_{-\mathbf{e}_i}(\mathbf{p}+\mathbf{e}_i; z)^{-1}$.

\begin{example}
\label{ex:cmf-construction}
Let us construct the CMF defined in \cref{def:the-cmf}, with $r=3$.

By the definition, we have $m = q = 3$ and $p = n = 2$.

First, we construct the theta matrix. For this parameterization the differential equation is
{\small
\[
\big(-\theta^{3} + \left(b_{1} + b_{2} + b_{3} - z\right) \theta^{2} + \left(a_{1} z + a_{2} z - b_{1} b_{2} - b_{1} b_{3} - b_{2} b_{3} - 2 z\right) \theta +(-  a_{1} a_{2} z + a_{1} z + a_{2} z + b_{1} b_{2} b_{3} - z)\big)F =0
\]
}
We extract our coefficients
\begin{itemize}
    \item $t_3(\mathbf{p}; z) = -1$
    \item $t_2(\mathbf{p}; z) = \left(b_{1} + b_{2} + b_{3} - z\right)$
    \item $t_1(\mathbf{p}; z) = \left(a_{1} z + a_{2} z - b_{1} b_{2} - b_{1} b_{3} - b_{2} b_{3} - 2 z\right)$
    \item $t_0(\mathbf{p}; z) = - a_{1} a_{2} z + a_{1} z + a_{2} z + b_{1} b_{2} b_{3} - z$
\end{itemize}

to construct 
\[
M_\theta(\mathbf{p}; z) =
\left[\begin{matrix}
0 & 0 & - a_{1} a_{2} z + a_{1} z + a_{2} z + b_{1} b_{2} b_{3} - z
\\1 & 0 & a_{1} z + a_{2} z - b_{1} b_{2} - b_{1} b_{3} - b_{2} b_{3} - 2 z
\\0 & 1 & b_{1} + b_{2} + b_{3} - z
\end{matrix}\right].
\]

We then construct the basis matrices. We only have two cases as $n = p = 2$ and $m = q = 3$:
\begin{align*}
&\mathcal{M}_{-\mathbf{e}_i} = M_\theta(\mathbf{p}; z) - (a_i - 1) I =
&&\left[\begin{matrix}
1 - a_{i} & 0 & - a_{1} a_{2} z + a_{1} z + a_{2} z + b_{1} b_{2} b_{3} - z\\
1 & 1 - a_{i} & a_{1} z + a_{2} z - b_{1} b_{2} - b_{1} b_{3} - b_{2} b_{3} - 2 z\\
0 & 1 & - a_{i} + b_{1} + b_{2} + b_{3} - z + 1\end{matrix}\right]
&& i \in \{1, 2\} \\ 
&\mathcal{M}_{\mathbf{e}_j} = -(M_\theta(\mathbf{p}; z) - b_{j-2} I) =
&&\left[\begin{matrix}
b_{j-2} & 0 & a_{1} a_{2} z - a_{1} z - a_{2} z - b_{1} b_{2} b_{3} + z\\
-1 & b_{j-2} & - a_{1} z - a_{2} z + b_{1} b_{2} + b_{1} b_{3} + b_{2} b_{3} + 2 z\\
0 & -1 & - b_{1} - b_{2} - b_{3} + b_{j-2} + z
\end{matrix}\right] && j \in \{3, 4, 5\}
\end{align*}

Finally, we can conclude by calculating $\mathcal{M}_{\mathbf{e}_i}(\mathbf{p}; z) = \mathcal{M}_{-\mathbf{e}_i}(\mathbf{p}+\mathbf{e}_i; z)^{-1}$.

\end{example}

\section{Dual Limit Convergence}
\label{sec:dual-limit-convergence}

In this section, we prove \cref{thm:dual-convergence}, the dynamical component of our result. The result is an application of Miller's algorithm \cite{gautschi1967computational} as described in \cref{lem:convergence-criteria}, enabled by the asymptotic behavior of the Meijer G-function within our specific parameterization.

\begin{definition}[Strongly Recessive Asymptotic]
\label{def:strongly-recessive}
Let $r \in \mathbb{N}, M \in \text{GL}_r(\mathbb{Q}(n))$ and consider the linear difference system:
$$\mathbf{x}(n+1)^T = \mathbf{x}(n)^T M(n)$$

If the system admits a unique recessive asymptotic (\cref{def:unique-recessive}) and the separation grows superpolynomially, then the system admits a \emph{strongly recessive asymptotic}. That is, for any integer $d>0$:
\[
\lim_{n\to\infty} n^d \frac{\|\mathbf{x}_r(n)\|}{\|\mathbf{x}_{r-1}(n)\|} = 0.
\]

\end{definition}

\subsection{Converging Trajectories}
\label{subsec:converging-trajectories}

Let \(L\) denote the right-loop contour from \cref{def:meijer-g}, and let
\(L'\) be its image under the change of variables \(x=n-s\), with the
orientation reversed.
Let \(a_1, a_2 \in \mathbb{Z}_{\le0}\) and write \(\alpha_i =1-a_i\) for \(i = 1, 2 \).

Define
\[
K(n,x)
=
\Gamma(x)^3
\Gamma(\alpha_1n-x+1)
\Gamma(\alpha_2n-x+1).
\]
and
\[
K_1(n,x)=K(n,x)e^{2\pi i x},\qquad
K_2(n,x)=K(n,x)e^{-2\pi i x},\qquad
K_3(n,x)=K(n,x).
\]

\[
\mathbf{I}_j(n)^{T}
=
\frac{1}{2\pi i}
\int_{L'}
\left(
1, n-x, (n-x)^2
\right)
K_j(n,x)\,dx,
\qquad j\in\{1,2,3\}.
\]

We first show that \(\mathbf{I}_1(n)^{T}, \mathbf{I}_2(n)^{T}, \mathbf{I}_3(n)^{T} \) are solutions of the same linear difference system and identify \(\mathbf I_3(n)^T\) as the Meijer \(G\)-function state vector \(\mathcal S_3(n\mathbf t)^T\) from \cref{def:the-cmf}.
We then evaluate the asymptotic growth of the three solutions and show that this solution is strongly recessive.

\begin{lemma}
\label{lem:pure-modes-satisfy}
For every \(\mathbf{t}=(a_1,a_2;1,1,1)\in\mathcal{T}_3^*\),
\[
\mathbf I_j(n+1)^{T}
=
\mathbf I_j(n)^{T}
\mathcal{M}_{\mathbf{0},\mathbf{t}}(n),
\qquad j\in\{1,2,3\}.
\]
Moreover,
\[
\mathbf I_3(n)^{T}
=
\mathcal S_3(n\mathbf t)^{T}.
\]
\end{lemma}

\begin{proof}
We prove the claim through auxiliary Meijer G-function branches. These branches are introduced only for the proof:
by the CMF construction they satisfy the desired trajectory recurrence, and by applying the reflection formula to their Mellin--Barnes integrands they become an invertible constant linear transformation of \(\mathbf{I}_1(n)^{T}, \mathbf{I}_2(n)^{T}, \mathbf{I}_3(n)^{T} \).

Define
\[
f_1(z;n)
=
G^{1,2}_{2,3}\!\left(
\begin{matrix}
a_1n,\ a_2n\\
n,\ n,\ n
\end{matrix}
\middle|z
\right), \quad
f_2(z;n)
=
(-1)^n
G^{2,2}_{2,3}\!\left(
\begin{matrix}
a_1n,\ a_2n\\
n,\ n,\ n
\end{matrix}
\middle|-z
\right),
\quad
f_3(z;n)
=
G^{3,2}_{2,3}\!\left(
\begin{matrix}
a_1n,\ a_2n\\
n,\ n,\ n
\end{matrix}
\middle|z
\right),
\]
and set
\[
\mathbf{x}_j(n)^{T}
=
\begin{pmatrix}
f_j(1;n)&
\theta f_j(1;n)&
\theta^2 f_j(1;n)
\end{pmatrix},
\qquad
\theta=z\frac{d}{dz}.
\]

We first compare these branches through the CMF construction of
\(\mathcal{M}_{\mathbf{0},\mathbf{t}}(n)\), as described in
\cref{sec:meijer-g-cmf}. For our Meijer \(G\)-parameterization
\(G^{m,r-1}_{r-1,r}\), the differential equation is
\[
\left[
(-1)^{-m} z
\prod_{i=1}^{p}(\theta-a_i+1)
-
\prod_{j=1}^{q}(\theta-b_j)
\right]G(z)=0.
\]
For \(m=1\) and \(m=3\), the sign of the \(z\)-term is negative. For \(m=2\),
the sign is positive, but the substitution \(z\mapsto -z\) changes this sign
while preserving the \(\theta\)-algebra:
\[
\theta_zG(-z)=\theta_wG(w)\big|_{w=-z}.
\]
Thus all three branches have the same effective differential equation and
therefore determine the same companion matrix \(M_\theta\).

The upper-parameter contiguous matrices agree exactly.
For the lower-parameter shifts, the unsigned matrices are \(M_\theta-b_jI\), while the signs depend on \(m\).
Along one step of \(\mathbf t=(a_1,a_2;1,1,1)\), each lower parameter is shifted once, so only the product of the three lower signs matters.
For a branch with parameter \(m\), this product is \((-1)^m\).
Hence the branches with \(m=1\) and \(m=3\) satisfy the same trajectory system.
The unnormalized middle branch has the opposite sign:
\[
\widetilde{\mathbf{x}}_2(n+1)^{T}
=
-\widetilde{\mathbf{x}}_2(n)^{T}
\mathcal{M}_{\mathbf{0},\mathbf{t}}(n).
\]
Since
\[
\mathbf{x}_2(n)^{T}
=
(-1)^n\widetilde{\mathbf{x}}_2(n)^{T},
\]
we obtain
\[
\mathbf{x}_2(n+1)^{T}
=
(-1)^{n+1}\widetilde{\mathbf{x}}_2(n+1)^{T}
=
(-1)^n\widetilde{\mathbf{x}}_2(n)^{T}
\mathcal{M}_{\mathbf{0},\mathbf{t}}(n)
=
\mathbf{x}_2(n)^{T}
\mathcal{M}_{\mathbf{0},\mathbf{t}}(n).
\]
Therefore, for all \(j\in\{1,2,3\}\),
\[
\mathbf{x}_j(n+1)^{T}
=
\mathbf{x}_j(n)^{T}
\mathcal{M}_{\mathbf{0},\mathbf{t}}(n).
\]

It remains to express the auxiliary branches in terms of \(\mathbf{I}_1(n)^{T}, \mathbf{I}_2(n)^{T}, \mathbf{I}_3(n)^{T} \).
Since \(q>p\), all three branches are represented using the same right-loop contour.
After the change of variables \(x=n-s\), applying \(\theta^\ell\) multiplies
the Mellin--Barnes integrand by
\[
s^\ell=(n-x)^\ell,
\]
so the state-vector identities follow componentwise from the scalar integrand
identities.
The upper Gamma factors become
\[
\Gamma(s-a_i n+1)
=
\Gamma((1-a_i)n-x+1)
=
\Gamma(\alpha_in-x+1),
\qquad i=1,2,
\]
and the common part of the three scalar integrands is precisely \(K(n,x)\).

We use the reflection identity
\[
\frac{1}{\Gamma(1-x)}
=
\frac{\Gamma(x)\sin(\pi x)}{\pi}.
\]
For the middle branch, the normalization \((-1)^n\) and the argument \(-z\) contribute
\[
(-1)^n(-1)^s
=
(-1)^ne^{\pi i s}
=
e^{-\pi i x}.
\]
Thus
\[
\frac{\sin(\pi x)}{\pi}e^{-\pi ix}
=
\frac{1-e^{-2\pi ix}}{2\pi i},
\qquad
\left(\frac{\sin(\pi x)}{\pi}\right)^2
=
\frac{2-e^{2\pi ix}-e^{-2\pi ix}}{4\pi^2}.
\]

Substituting these scalar multipliers into the three branch integrands gives
\[
\mathbf{x}_3(n)^{T}
=
\mathbf I_3(n)^{T},
\]
\[
\mathbf{x}_2(n)^{T}
=
\frac{1}{2\pi i}
\left(
-\mathbf I_2(n)^{T}
+
\mathbf I_3(n)^{T}
\right),
\]
and
\[
\mathbf{x}_1(n)^{T}
=
\frac{1}{4\pi^2}
\left(
-\mathbf I_1(n)^{T}
-\mathbf I_2(n)^{T}
+
2\mathbf I_3(n)^{T}
\right).
\]
Equivalently,
\[
\begin{pmatrix}
\mathbf{x}_1(n)^{T}\\
\mathbf{x}_2(n)^{T}\\
\mathbf{x}_3(n)^{T}
\end{pmatrix}
=
A
\begin{pmatrix}
\mathbf{I}_1(n)^{T}\\
\mathbf{I}_2(n)^{T}\\
\mathbf{I}_3(n)^{T}
\end{pmatrix},
\]
where
\[
A=
\begin{pmatrix}
-\frac{1}{4\pi^2} & -\frac{1}{4\pi^2} & \frac{1}{2\pi^2}\\[0.8em]
0 & -\frac{1}{2\pi i} & \frac{1}{2\pi i}\\[0.8em]
0 & 0 & 1
\end{pmatrix},
\qquad
\det A=\frac{1}{8\pi^3 i}\neq0.
\]
Since \(A\) is constant and invertible,  \(\mathbf{I}_1(n)^{T}, \mathbf{I}_2(n)^{T}, \mathbf{I}_3(n)^{T} \) are constant linear combinations of the auxiliary branch solutions.
Therefore they satisfy the same trajectory system:
\[
\mathbf I_j(n+1)^{T}
=
\mathbf I_j(n)^{T}
\mathcal{M}_{\mathbf{0},\mathbf{t}}(n),
\qquad j\in\{1,2,3\}.
\]

Finally, the third auxiliary branch is the canonical branch used in \cref{def:the-cmf}, and the comparison above gives
\[
\mathcal S_3(n\mathbf t)^{T}
=
\mathbf{x}_3(n)^{T}
=
\mathbf I_3(n)^{T}.
\]
\end{proof}

\begin{lemma}
\label{lem:asymptotic-extraction}
Let
\[
\rho_j=e^{-2\pi i j/3},\qquad
H(n)=\Gamma(\alpha_1 n+1)\Gamma(\alpha_2 n+1),
\qquad
c=(\alpha_1\alpha_2)^{1/3} > 0.
\]
Then, for \(j=1,2,3\),
\[
\|\mathbf I_j(n)\|
=
H(n)
\exp\left(
-3c\,\operatorname{Re}(\rho_j)n^{2/3}
+\mathcal{O}(n^{1/3})
\right).
\]

Moreover,
\(
\mathbf I_1,\mathbf I_2,\mathbf I_3
\)
are linearly independent.

\end{lemma}

The proof is given in \cref{proof:asymptotic-extraction}

\begin{lemma}[Recessive Trajectories]
\label{lem:recessive-trajectories}
Let \(\mathbf{t}\in\mathcal{T}_3^*\). Then the difference system defined by
\(
\mathcal{M}_{\mathbf{0},\mathbf{t}}(n)
\)
admits a strongly recessive solution represented by the state vector
\(
\mathcal{S}_3(n\mathbf{t})^T
\)
defined in \cref{def:the-cmf}.
\end{lemma}

\begin{proof}
By \cref{lem:pure-modes-satisfy},
\(
\mathbf I_1(n)^{T},
\mathbf I_2(n)^{T},
\mathbf I_3(n)^{T}
\)
are solutions of the trajectory system, and by
\cref{lem:asymptotic-extraction} they differ asymptotically and are linearly independent. Thus, they form a Birkhoff-Trjitzinsky basis of solutions. 
Moreover,
\[
\operatorname{Re}(\rho_1)
=
\operatorname{Re}(\rho_2)
=
-\frac12,
\qquad
\operatorname{Re}(\rho_3)=1.
\]
Therefore, for \(j=1,2\),
\[
\frac{\|\mathbf I_3(n)\|}{\|\mathbf I_j(n)\|}
=
\exp\left(
-\frac{9c}{2}n^{2/3}
+
\mathcal O(n^{1/3})
\right),
\]
which decays faster than any power of \(n\).
Thus \(\mathbf I_3(n)^{T}\) is the strongly recessive solution in
this basis.

Finally, by \cref{lem:pure-modes-satisfy},
\[
\mathbf I_3(n)^{T}
=
\mathcal S_3(n\mathbf t)^{T}.
\]
Hence \(\mathcal S_3(n\mathbf t)^{T}\) is a strongly recessive
solution of the difference system defined by
\(\mathcal M_{\mathbf0,\mathbf t}(n)\).
\end{proof}

\subsection{Rational Gauge Invariance}

\begin{lemma}
\label{lem:gauge-invariance-recessive}
Let $\mathbf{x}(n+1)^T = \mathbf{x}(n)^T M(n)$ be a linear difference system admitting a strongly recessive asymptotic (\cref{def:strongly-recessive}) denoted by $\mathbf{x}_r(n)^T$. 

Let $A(n) \in \mathrm{GL}_r(\mathbb{Q}(n))$ be a rational gauge transformation matrix, defining the gauge-equivalent system $C(n)=A(n)^{-1}M(n) A(n+1)$.

Then, the gauged system $C(n)$ admits a strongly recessive asymptotic equal to $\mathbf{x}_r(n)^TA(n)$.
\end{lemma}

\begin{proof}
For all sufficiently large $n$, the rational gauge matrix $A(n)$ defines an invertible correspondence between the solution spaces of the two systems.
Let $\mathbf{x}_i(n)$ be the primal basis solutions of $M(n)$. The corresponding primal solutions of the gauged system $C(n)$ are given by the mapping:
\[
\mathbf{z}_i(n)^T = \mathbf{x}_i(n)^T A(n)
\]

Because $A(n) \in \mathrm{GL}_r(\mathbb{Q}(n))$, its induced matrix norm is bounded by some positive constants $C_1, C_2 \in \mathbb{N}$ and polynomial degrees $d_1, d_2 \in \mathbb{N}$ such that for all sufficiently large \(n\):
\[
\|A(n)\| \le C_1 n^{d_1} \quad \text{and} \quad \|A(n)^{-1}\| \le C_2 n^{d_2}
\]

By the sub-multiplicativity of induced matrix norms:
\[
\|\mathbf{z}_i(n)^T\| = \|\mathbf{x}_i(n)^T A(n) \| \le \|\mathbf{x}_i(n)^T \| \cdot \| A(n)\| \le C_1 n^{d_1} \|\mathbf{x}_i(n)^{T}\|
\]

Similarly:
\[
\|\mathbf{x}_{i}(n)^T\| =  \|\mathbf{z}_{i}(n)^T A(n)^{-1} \| \le \|\mathbf{z}_{i}(n)^T\| \cdot \|A(n)^{-1}\| \le C_2 
n^{d_2} \|\mathbf{z}_{i}(n)^T\|
\]

Denoting $C=\max(C_1, C_2)$ and $d = \max(d_1, d_2)$, we obtain the bounds:
\[
\frac{1}{C\cdot n^d} \| \mathbf{x}_i(n)\| \le \|\mathbf{z}_{i}(n)\| \le C \cdot n^d \|\mathbf{x}_i(n)\|
\]

Finally, the gauged system satisfies the strongly recessive criteria. For any $b\in\mathbb{N}$, we have
\[
0 \le n^b\frac{\|\mathbf{z}_r(n)\|}{\|\mathbf{z}_{i}(n)\|} \le C^2n^{2d+b} \frac{\|\mathbf{x}_r(n)\|}{\|\mathbf{x}_{i}(n)\|} \xrightarrow[n\to\infty]{} 0, \qquad i<r
\]

The upper bound unconditionally vanishes due to the assumption that $\mathbf{x}_r(n)$ is a strongly recessive asymptotic (\cref{def:strongly-recessive}). Therefore, the mapped state $\mathbf{z}_r(n)$ strictly satisfies the definition of a strongly recessive asymptotic for the gauged system $C(n)$. 
\end{proof}

\begin{corollary}
\label{cor:gauge-invariance-limit}
Let $\mathcal{S}_r(0) = \mathbf{x}_r(0)$ be the initial state of the strongly recessive asymptotic of $\mathcal{M}_{\mathbf{0}, \mathbf{t}}$.

Let $U(n)\in\mathrm{GL}_r(\mathbb Q(n))$ be a rational gauge that is regular
and invertible for every $n\in\mathbb N_0$, and define
\[
C(n)=U(n)^{-1}\mathcal M_{\mathbf0,\mathbf t}(n)^{-T}U(n+1).
\]

For any vector \(\mathbf v\in\mathbb R^r\) generic with respect to the gauge-equivalent dual system \(C(n)\) (\cref{def:generic-vector}), the forward iteration isolates the gauged initial state up to scaling.
That is, there exists a sequence $c_n\in\mathbb{R}\setminus\{0\}$ such that

\[
\lim_{n\to\infty}
c_n\Phi_C(n)\mathbf{v}
=
U(0)^{-1}\mathcal{S}_r(\mathbf{0}).
\]

\end{corollary}

\begin{proof}

By \cref{lem:gauge-invariance-recessive}, the system $C(n)^{-T}=U(n)^T \mathcal{M}_{\mathbf{0}, \mathbf{t}}(n)U(n+1)^{-T}$ admits the strongly recessive asymptotic $\mathcal{S}_r(n\mathbf{t})U(n)^{-T}$.

The result is then achieved directly by \cref{lem:convergence-criteria}.
\end{proof}

\subsection{Proof of~\texorpdfstring{\Cref{thm:dual-convergence}}{Theorem~\ref{thm:dual-convergence}}}

\begin{proof}[Proof of \cref{thm:dual-convergence}]
Let \(r=3\) and \(\mathbf{t}\in \mathcal{T}^*_{r=3}\).

By \cref{lem:recessive-trajectories},
\(\mathcal S_r(n\mathbf t)^{T}\) is a strongly recessive solution of the system defined by $\mathcal{M}_{\mathbf{0}, \mathbf{t}}(n)$.

From \cref{lem:convergence-criteria}, for any generic vector $\mathbf{v} \in \mathbb{R}^r$, the transition matrix of the dual system isolates the unique recessive solution at $n=0$ up to scaling. Hence, there exists a sequence $c_n\in\mathbb{R}\setminus\{0\}$ such that
\[
\lim_{n\to\infty}
c_n\Phi_{\mathcal{M}_{\mathbf{0},\mathbf{t}}}(n)^{-T}\mathbf{v}
=
\mathcal{S}_{r}(\mathbf{0}).
\]

Finally, conversion to an Ap\'ery-like recurrence by the CVM algorithm (\cref{def:cyclic-vector-method}) is a rational gauge transformation, so the same isolation is preserved by \cref{cor:gauge-invariance-limit} whenever the gauge is regular and invertible on \(\mathbb{N}_0\).
\end{proof}

\subsection{Example}
\label{sec:saddle-point-examples}

We illustrate the preceding construction for the trajectory with upper
parameters \((0,0)\) and lower parameters \((1,1,1)\), namely \( \mathbf t=(0,0;1,1,1) \).
This yields the following trajectory matrix of order 3:
\[
\mathcal{M}_{\mathbf{0}, \mathbf{t}}(n) = \begin{bmatrix}
    1 & n^{3} & (n-1)(3n+3)(n^2+n+1) \\
    2 & 1 - 3 n^{2} & - (n + 1) (8 n^{2} + n + 5)\\
    1 & 3 n + 2 & 6 n (n + 1)
\end{bmatrix}.
\]

In this case \(a_1=a_2=0\), so \( c=\big((1-a_1)(1-a_2)\big)^{1/3}=1 \) and \( H(n) = \Gamma(n+1)^2 = (n!)^2 \).

By \cref{lem:asymptotic-extraction}, we have:
\[
\|\mathbf I_1(n)\|,
\|\mathbf I_2(n)\|
=
(n!)^2
\exp\left(
\frac32 n^{2/3}
+
\mathcal O(n^{1/3})
\right), \qquad
\|\mathbf I_3(n)\|
=
(n!)^2
\exp\left(
-3 n^{2/3}
+
\mathcal O(n^{1/3})
\right).
\]
Hence \(\mathbf I_3(n)\) is superpolynomially smaller than the two dominant conjugate solutions.
Since $\mathcal S_3(n\mathbf t)^T=\mathbf I_3(n)^T$, the state vector is the strongly recessive solution of the trajectory system.

This agrees with the asymptotic splitting obtained computationally from the
gauged companion system using RISC's Asymptotics \cite{kauers2011mathematica}.
By \cref{lem:gauge-invariance-recessive}, conversion to companion form by CVM
preserves the exponential scales up to polynomial factors.

The solver returns the three exponential phases
\begin{align*}
u_1(n) &= n^{2n} e^{-2n} \exp\left(-3 n^{2/3} + \mathcal{O}(n^{1/3})\right) \\
u_2(n) &= n^{2n} e^{-2n} \exp\left(-3 \omega n^{2/3} + \mathcal{O}(n^{1/3})\right) \\
u_3(n) &= n^{2n} e^{-2n} \exp\left(-3 \omega^2 n^{2/3} + \mathcal{O}(n^{1/3})\right) 
\end{align*}
where \(\omega = e^{2\pi i / 3}\).

The common factor \(n^{2n}e^{-2n}\) agrees with \((n!)^2\) up to a polynomial factor by Stirling's formula.
Thus the computational output matches \cref{lem:asymptotic-extraction}, up to the ordering of the basis.

\section{Examples}
\label{sec:examples}

In this section, we demonstrate the full sequence generation pipeline outlined in \cref{fig:algorithm-pipeline}.

All examples presented here, along with the automated generation framework, are available in the \href{https://colab.research.google.com/drive/189rWd7_7M1i2d5E03OoMwGTWVuG1V75l?usp=sharing}{supplementary Colab notebook}.
The implementation relies on the \texttt{ramanujantools} package \cite{https://doi.org/10.5281/zenodo.15747099}, utilizing python-flint \cite{hart2013flint} to optimize heavy polynomial arithmetic.

To illustrate the mechanics of the Conservative Matrix Field, we focus on the CMF case $\mathcal{M} = \mathcal{M}_3$, as defined in \cref{def:the-cmf}.
The CMF is of rank $r=3$ and lattice dimension $2r-1=5$, which yields holonomic sequences of order 3.

We fix our source point at the origin $\mathbf{0} \in \mathbb{Z}^5$.
As derived in \cref{prop:g-function-identity} and explicitly verified in \cref{ex:identity-2}, any trajectory matrix $\mathcal{M}_{\mathbf{0}, \mathbf{t}}(n)$ from this source that satisfies strict dual convergence criteria (\cref{lem:convergence-criteria}) will converge to $\gamma$ when projected onto the following initial value matrix:
\[
\begin{bmatrix} p(-3) & p(-2) & p(-1) \\ q(-3) & q(-2) & q(-1) \end{bmatrix} = \begin{bmatrix} 1 & 1 & 0 \\ 1 & 1 & 1 \end{bmatrix}
\]

\subsection{Full Algorithm Showcase}
\label{subsec:full-algorithm}

Let $\mathbf{t} = (0, 0; 1, 1, 1)^T$. This trajectory lies within the verified converging family (\cref{lem:recessive-trajectories}) therefore convergence is mathematically guaranteed. We calculate the trajectory matrix:
\[
\mathcal{M}_{\mathbf{0}, \mathbf{t}}(n) = \left[\begin{matrix}
    1 & n^{3} & (n-1)(3n+3)(n^2+n+1) \\
    2 & 1 - 3 n^{2} & - \left(n + 1\right) \left(8 n^{2} + n + 5\right)\\
    1 & 3 n + 2 & 6 n \left(n + 1\right)
\end{matrix}\right]
\]

We then calculate the dual matrix $\mathcal{M}_{\mathbf{0}, \mathbf{t}}(n)^{-T}$ (normalizing rational functions to polynomials):
\[
\mathcal{M}_{\mathbf{0}, \mathbf{t}}(n)^{-T} = \frac{1}{(n+1)^4}\left[\begin{matrix}
6 n^{2} + 13 n + 10 & - 8 n - 5 & 3 \\
3 n^{3} + 3 n^{2} - 3 n - 6 & - 3 n^{2} + 3 n + 3 & n - 2 \\
n^{4} - 2 n^{3} - 6 n^{2} - 3 n + 3 & 6 n^{2} + 2 n - 1 & 1 - 2 n
\end{matrix}\right]
\]

Applying the cyclic vector method (\cref{def:cyclic-vector-method}) retrieves the following gauge matrix:
\[
U(n) =
\begin{bmatrix}1 & 6 n^{2} + 13 n + 10 & 15 n^{4} + 123 n^{3} + 385 n^{2} + 540 n + 284\\0 & 3 n^{3} + 3 n^{2} - 3 n - 6 & 10 n^{5} + 66 n^{4} + 143 n^{3} + 52 n^{2} - 183 n - 169\\0 & n^{4} - 2 n^{3} - 6 n^{2} - 3 n + 3 & 6 n^{6} + 29 n^{5} + 18 n^{4} - 119 n^{3} - 209 n^{2} - 33 n + 83\end{bmatrix}
\]

Applying the gauge transform on the normalized matrix $(n+1)^4 \mathcal{M}_{\mathbf{0}, \mathbf{t}}(n)^{-T}$ yields a companion matrix representing the following holonomic sequence:
\begin{align*}
0 &= p{\left(n - 3 \right)} \cdot \left(n + 1\right)^{4} \left(n + 2\right)^{2} \left(8 n^{7} + 99 n^{6} + 504 n^{5} + 1350 n^{4} + 2013 n^{3} + 1590 n^{2} + 532 n + 27\right) \\
& - p{\left(n - 2 \right)} \cdot \left(n + 2\right)^{2} \big(24 n^{9} + 273 n^{8} + 1246 n^{7} + 2883 n^{6} + 3408 n^{5} + 1072 n^{4} - 3051 n^{3} - 5553 n^{2} \\
& \qquad - 4181 n - 1242\big) \\
& + p{\left(n - 1 \right)} \cdot \left(24 n^{9} + 313 n^{8} + 1616 n^{7} + 4059 n^{6} + 4623 n^{5} + 435 n^{4} - 4470 n^{3} - 4923 n^{2} - 1749 n + 243\right)\\
& + p{\left(n \right)} \cdot \left(- 8 n^{7} - 43 n^{6} - 78 n^{5} - 35 n^{4} + 47 n^{3} + 72 n^{2} + 27 n - 9\right)
\end{align*}

Using the telescopic relation provided by the CVM gauge matrix $U(n)$, we transform our initial conditions to fit this specific scalar sequence:
\[
\resizebox{\textwidth}{!}{$
\begin{bmatrix} p'(-3) & p'(-2) & p'(-1) \\ q'(-3) & q'(-2) & q'(-1) \end{bmatrix} =
\begin{bmatrix} p(-3) & p(-2) & p(-1) \\ q(-3) & q(-2) & q(-1) \end{bmatrix} U(0) = 
\left[\begin{matrix}
    1 & 1 & 0 \\
    1 & 1 & 1\end{matrix}\right]
\cdot
\left[\begin{matrix}
    1 & 10 & 284 \\
    0 & -6 & -169 \\
    0 & 3 & 83
\end{matrix}\right] =
\left[\begin{matrix}
    1 & 4 & 115 \\
    1 & 7 & 198
\end{matrix}\right]
$}
\]

The holonomic sequences $p'(n)$ and $q'(n)$ defined by these initial values satisfy
\[\lim_{n\to\infty}\frac{p'(n)}{q'(n)}=\gamma.\]

\subsection{Trajectories Violating the Convergence Criteria} 
To understand the limits of this approach, we consider three trajectories that violate the requirements of the convergence criteria (\cref{lem:convergence-criteria}).
In order to verify that these trajectories do not satisfy it, we use the Birkhoff-Trjitzinsky engine \cite{birkhoff1933analytic} of the \texttt{ramanujantools} pacakge \cite{https://doi.org/10.5281/zenodo.15747099}, which was validated against RISC's \texttt{Asymptotics} package~\cite{kauers2011mathematica}.

To run it, we convert our forward system into a companion form using the CVM algorithm~\cref{def:cyclic-vector-method}.
As defined in~\cref{sec:meijer-g-cmf}, our trajectory matrix $\mathcal{M}_{\mathbf{0}, \mathbf{t}}(n)$ satisfies the following forward system:

$$\mathcal{S}_\mathbf{t}(n+1)^T = \mathcal{S}_\mathbf{t}(n)^T\mathcal{M}_\mathbf{0, t}(n),$$ 
where
$$\mathcal{S}_\mathbf{t}(n) := S_3(n\mathbf{t}; 1)$$
and
\[
\mathcal{S}_3(n\mathbf{t};z) := [G^{3, 2}_{2, 3}(n\mathbf{t}; z), \theta G^{3, 2}_{2, 3}(n\mathbf{t} ;z), \dots, \theta^{r-1}G^{3, 2}_{2, 3}(n\mathbf{t} ; z)]
\]

The CVM algorithm constructs a gauge matrix $U(n)$ such that $C(n) = U(n)^{-1}\mathcal{M}_{\mathbf{s, t}}(n)U(n+1)$ is a companion matrix representing a scalar recurrence.
Mapping a forward-looking CVM equation into a standard backward-looking companion form implicitly aliases the base state to the highest forward index, thus the natural solution to this derived recurrence is our Meijer G-function parameterization, reflecting an index shift equal to the order of the system:
$$p(n) = G^{3, 2}_{2, 3}((n+3)\mathbf{t}; 1).$$

Because the gauge matrix $U(n)$ is rational in $n$, the asymptotic separation hierarchy is preserved when converting the system if the separation (of the unique recessive solution) is exponential or larger, which is the case in all three examples.

\subsubsection{Example 1}

Let $\mathbf{t}_1 = (0, 0; 0, 0, 1)^T$, corresponding to the trajectory matrix

\[
\mathcal{M}_{\mathbf{0}, \mathbf{t}_1}(n) =
\left[\begin{matrix}
    n^{2} & -1 & 0 \\
    - 2 n & -2 & \left(n - 1\right) \left(n + 1\right) \\
    1 & -1 & - 2 n - 2
\end{matrix}\right]
\]

Examining the asymptotics of the gauged companion form:

\[
u_1(n) = n^{2} \log{\left(n \right)} n!, \quad
u_2(n) = n^{2} n!, \quad
u_3(n) = 1
\]

While this system does have a unique recessive solution, it is definitely not our state vector, as
\[\lim_{n\to\infty}{G^{3, 2}_{2, 3}(n\mathbf{t_1}; 1)} = \infty.\]

Therefore, the dual system is expected to converge to a different constant, as seen in~\cref{fig:convergence_contrast}.

\subsubsection{Example 2}

Let $\mathbf{t}_2 = (0, 0; 0, 1, 1)^T$, corresponding to the trajectory matrix

\[
\mathcal{M}_{\mathbf{0}, \mathbf{t}_2}(n) =
\left[\begin{matrix}
    2n & -\left(n - 1\right) \left(2 n^{2} + 2 n + 1\right) & -\left(n + 1\right) \left(4 n^{3} + 2 n - 3\right) \\
    2 n - 1 & 5 n^{2} + 3 n + 1 & - \left(- 2 n - 2\right) \left(4 n^{2} + 3\right) \\
    -1 & - 2 n - 1 & - \left(n + 1\right) \left(3 n - 2\right)
\end{matrix}\right]
\]

Examining the asymptotics of the gauged companion form:
\[
u_1(n) = n^{4} \left(n!\right)^{2}, \quad
u_2(n) = \left(-1\right)^{n} n^{\frac{11}{4}} e^{2 i \sqrt{n}} n!, \quad
u_3(n) = \left(-1\right)^{n} n^{\frac{11}{4}} e^{- 2 i \sqrt{n}} n!
\]

Because the primal system has two conjugate recessive solutions, the dual system is expected to oscillate, as seen in \cref{fig:convergence_contrast}.

\subsubsection{Example 3}

Let $\mathbf{t}_3 = (0, 0; 1, 2, 2)^T$, corresponding to the trajectory matrix
\begin{equation*}
\resizebox{\textwidth}{!}{
$\displaystyle
\mathcal{M}_{\mathbf{0}, \mathbf{t}_3}(n) = 
\left[\begin{matrix}
    4 n^{4} + 4 n^{3} + 4 n^{2} - n - 2 & n \left(4 n^{4} - 12 n^{3} - 12 n^{2} + 3 n + 2\right) & \left(n + 1\right) \left(4 n^{5} - 80 n^{4} - 64 n^{3} - n^{2} + 23 n + 14\right) \\
    - 4 n^{3} - 6 n - 5 & - 4 n^{4} + 32 n^{3} + 34 n^{2} - n - 4 & \left(n + 1\right) \left(- 4 n^{4} + 148 n^{3} + 130 n^{2} + 49 n + 21\right) \\
    n^{2} - 4 n - 3 & n^{3} - 20 n^{2} - 27 n - 8 & - \left(n + 1\right) \left(- n^{3} + 69 n^{2} + 55 n + 5\right)
\end{matrix}\right]
$}
\end{equation*}

Examining the asymptotics of the gauged companion form:
\[
u_1(n) = n^{10} \left(n!\right)^{4}, \quad
u_2(n) = \left(-16\right)^{n} n^{\frac{31}{4}} e^{4 i \sqrt{n}} \left(n!\right)^{3}, \quad
u_3(n) = \left(-16\right)^{n} n^{\frac{31}{4}} e^{- 4 i \sqrt{n}} \left(n!\right)^{3}
\]

Similarly to the previous example, the primal system has two conjugate recessive solutions, therefore the dual system is expected to oscillate, as seen in \cref{fig:convergence_contrast}.

\begin{figure}[H]
    \centering
    \includegraphics[width=\textwidth]{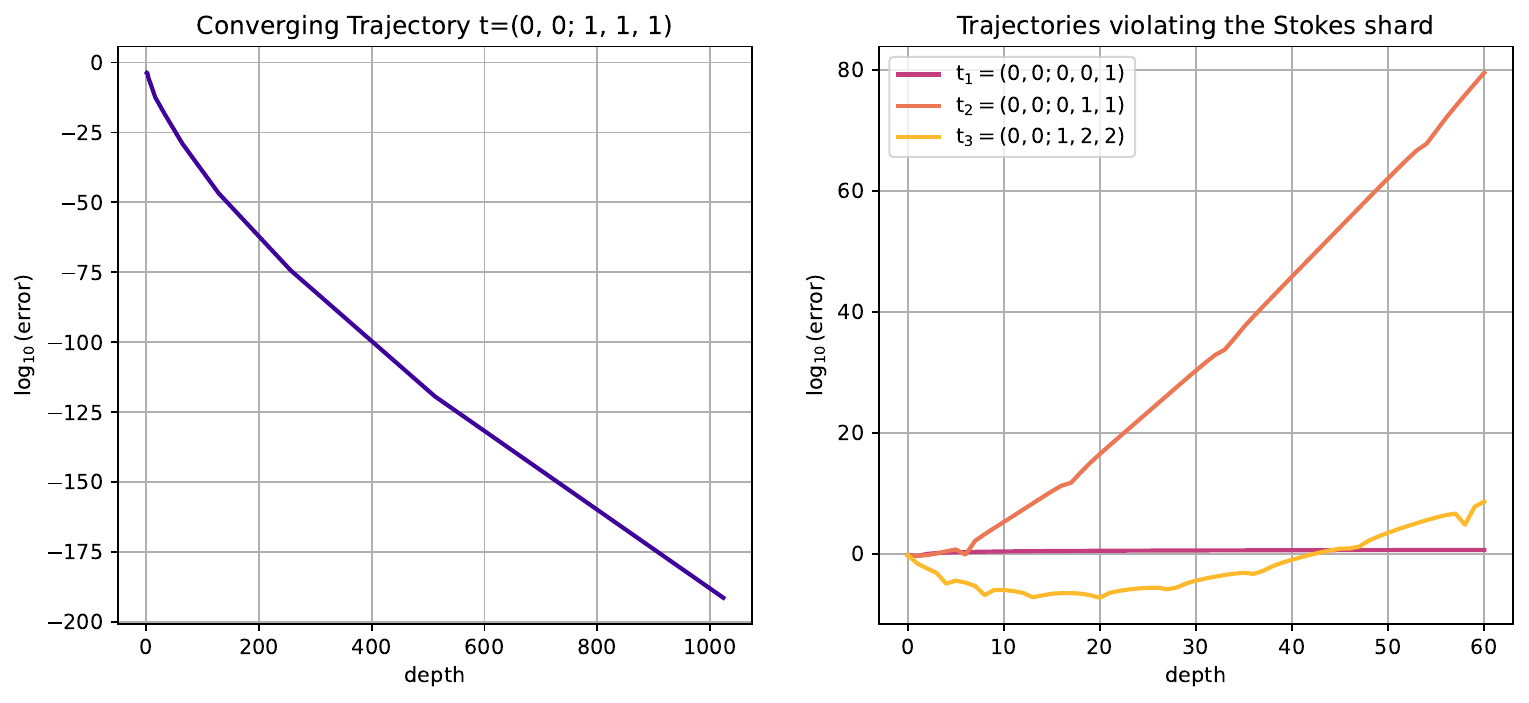}
    \caption[Dichotomy of convergence for $p=2$ sequences.]{
        \textbf{Dichotomy of convergence for $p=2$ sequences.}
        \textbf{(Left)} The converging trajectory $\mathbf{t}=(0, 0; 1, 1, 1)^T$ from example 1, showing logarithmic error decreasing smoothly to $10^{-200}$ at depth $n=1000$.
        \textbf{(Right)} The trajectories violating the convergence criteria: $\mathbf{t}_1=(0, 0; 0, 0, 1)$ converges to a different constant; $\mathbf{t}_2 = (0, 0; 0, 1, 1)$ and $\mathbf{t}_3=(0, 0; 1, 2, 2)$ oscillate and diverge.
    }
    \label{fig:convergence_contrast}
\end{figure}

\subsection{Convergence Rate Comparison}
\label{subsec:examples-comparison}

We now explore trajectories strictly inside the converging family $\mathcal{T}^*_3 \subset \mathbb{Z}^3$ (\cref{lem:recessive-trajectories}).
\Cref{fig:compare} illustrates the convergence rate of these sequences—measured as logarithmic error decay—evaluated both by trajectorial steps (left) and by lattice steps, defined by the $\ell_1$ norm of the trajectory (right).
The color progression highlights the influence of the trajectory parameters on the approximation quality: the convergence rate per step increases with the magnitude of $|a_i|$ (indicated by the lighter curves) and diminishes with larger $|b_j|$ (indicated by the darker curves).

We plot the recurrences by Rivoal \cite{rivoal2009rational} and by Aptekarev et al. recurrence \cite{aptekarev2009linear} (dashed bright lines) as reference points. They are omitted from the normalized plot as neither corresponds to a trajectory in the CMF lattice.
Their approximations exhibit a convergence rate similar to that of existing trajectories.

Following \cite{elimelech2024algorithm,weinbaum2025conservative}, we define the numerical approximation of the irrationality measure at depth $n$ as
$$\delta = -1 - \log_{\left|q_n\right|}\left(\left|\frac{p_n}{q_n}-\gamma\right|\right),$$
where $\frac{p_n}{q_n}$ is the corresponding rational approximation.
Regardless of the chosen trajectory, empirical evaluations of $\delta$ consistently converge to $-1$.
As this represents the theoretical lower bound for $\delta$, the sequences generated by this framework do not directly contribute to an irrationality proof for $\gamma$.

\begin{figure}[H] % h = here, t = top, b = bottom, p = page
    \centering
    \includegraphics[width=1\textwidth]{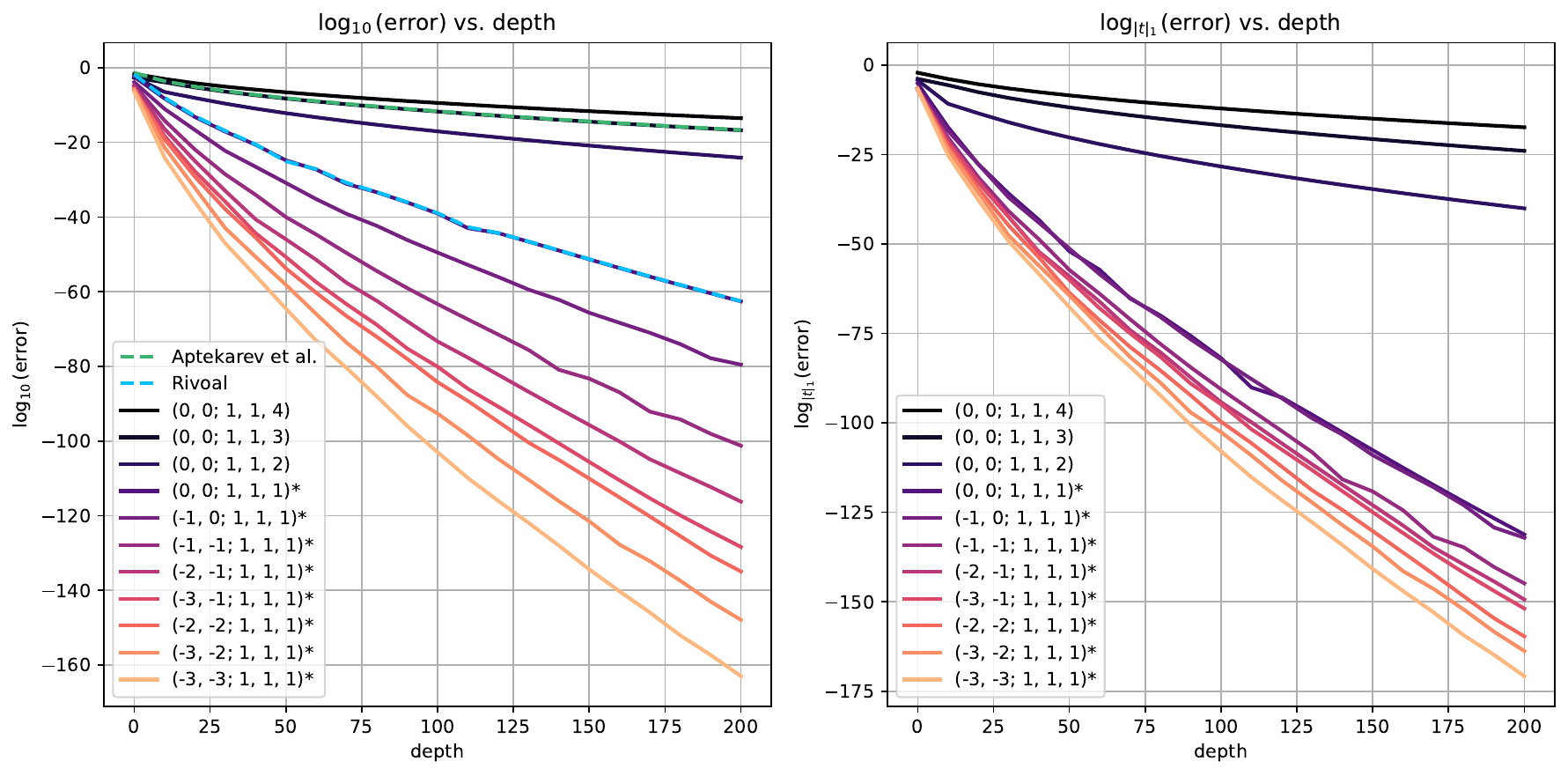}
    \caption{
    \textbf{A selection of formulas for $\gamma$ derived from trajectories in $\mathcal{T}_3$, compared by convergence rate}. Trajectories in $\mathcal{T}^*_3$, proven to converge, are marked by an asterisk.
    \textbf{(Left)} The formulas' convergence rate, presenting the number of digits
    $\log_{10}(\textup{err})$ as a function of depth. The recurrences by Rivoal and by Aptekarev et al. are highlighted as references. \textbf{(Right)} The formulas' normalized convergence rate, to reflect the complexity of each step in the trajectory.
    $\log_{|\mathbf{t}|_1}(\textup{err})$ as a function of depth.
    }
    \label{fig:compare}
\end{figure}

\vspace{-8pt}
\subsection{A connection to Rivoal's recurrence}
\label{subsec:rivoal-equivalence}

The numerical comparison in \cref{subsec:examples-comparison} suggests that the $\mathbf{t} = (0, 0; 1, 1, 1)^T$ trajectory and Rivoal's recurrence are related due to their similar convergence rate. 
Indeed, they are connected by a gauge transform up to a scalar factor.
Denoting our construction (from \cref{subsec:full-algorithm})
\[
\mathcal{M}_{\mathbf{0}, \mathbf{t}}(n)^{-T} = \frac{1}{(n+1)^4}\left[\begin{matrix}
6 n^{2} + 13 n + 10 & - 8 n - 5 & 3 \\
3 n^{3} + 3 n^{2} - 3 n - 6 & - 3 n^{2} + 3 n + 3 & n - 2 \\
n^{4} - 2 n^{3} - 6 n^{2} - 3 n + 3 & 6 n^{2} + 2 n - 1 & 1 - 2 n
\end{matrix}\right],
\]

Rivoal's recurrence \cite{rivoal2009rational} is:
\begin{align*}
0 &= \left(n + 2\right)^{2} \left(8 n + 19\right) \left(8 n + 27\right) p{\left(n - 3 \right)} 
  - \left(8 n + 27\right) \left(24 n^{3} + 105 n^{2} + 124 n + 25\right) p{\left(n - 2 \right)} \\
  & + \left(n + 3\right) \left(8 n + 11\right) \left(24 n^{2} + 145 n + 215\right) p{\left(n - 1 \right)} - \left(n + 3\right)^{2} \left(8 n + 11\right) \left(8 n + 19\right) p{\left(n \right)}
\end{align*}
which translates to the companion form:
\[
C(n)=
\begin{bmatrix}0 & 0 & \frac{\left(n + 2\right)^{2} \left(8 n + 27\right)}{\left(n + 3\right)^{2} \left(8 n + 11\right)}\\1 & 0 & - \frac{\left(8 n + 27\right) \left(24 n^{3} + 105 n^{2} + 124 n + 25\right)}{\left(n + 3\right)^{2} \left(8 n + 11\right) \left(8 n + 19\right)}\\0 & 1 & \frac{24 n^{2} + 145 n + 215}{\left(n + 3\right) \left(8 n + 19\right)}
\end{bmatrix}.
\]

Let 
\[
G(n) =
\begin{bmatrix}0 & \frac{8 n + 11}{\left(n + 1\right)^{2}} & \frac{24 n^{2} + 97 n + 94}{\left(n + 1\right)^{2} \left(n + 2\right)}\\-1 & \frac{3 n^{2} - n - 7}{\left(n + 1\right)^{2}} & \frac{15 n^{4} + 69 n^{3} + 66 n^{2} - 78 n - 111}{\left(n + 1\right)^{2} \left(n + 2\right)^{2}}\\2 & \frac{- 6 n^{2} - 6 n + 3}{\left(n + 1\right)^{2}} & \frac{8 n^{5} + 21 n^{4} - 38 n^{3} - 131 n^{2} - 48 n + 53}{\left(n + 1\right)^{2} \left(n + 2\right)^{2}}
\end{bmatrix},
\]
then
\[
C = (n+1)^2\cdot G^{-1}(n)\cdot\mathcal{M}_{\mathbf{0}, \mathbf{t}}(n)^{-T}\cdot G(n+1).
\]

We thank Pilar Stolarczyk for observing a similar connection between Rivoal's recurrence and a modified version of our construction, as part of a submission to the Ramanujan Challenge for AI \cite{shalyt2026ramanujan}.

\section{Conclusion}

The framework developed in this work conjectures that for $r \ge 3$, there exists two initial value vectors $\mathbf{p}_r,\mathbf{q}_r \in \mathbb{Z}^r$ and an infinite set of trajectories $\mathcal{T}_r\subset \mathbb{Z}^{2r-1}$ such that for any $\mathbf{t} \in \mathcal{T}^*_r$ the dual CMF limit converges to Euler's constant:
\[
\mathcal{D}_{\mathbf{0}, \mathbf{t}}^{\mathcal{M}}(\mathbf{p}_r,\mathbf{q}_r,\mathbf{v}) =
\lim_{n\to\infty} \frac
{\mathbf{p}_r^T \cdot \Phi_{\mathcal{M}_{\mathbf{0}, \mathbf{t}}}(n)^{-T}\cdot \mathbf{v}}
{\mathbf{q}_r^T\cdot\Phi_{\mathcal{M}_{\mathbf{0}, \mathbf{t}}}(n)^{-T}\cdot \mathbf{v}}
= \gamma.
\]

This result is a composition of two mathematical mechanisms:
\begin{enumerate}
    \item \textbf{The Algebraic Component:} 
    The initial value vectors $\mathbf{p}_r$ and $\mathbf{q}_r$ are derived from the Meijer G-function identity (\cref{prop:g-function-identity}) to satisfy the exact evaluation, proven for any $r\ge3$:
    \[\frac{\mathbf{p}_r^T\mathcal{S}(0)}{\mathbf{q}_r^T\mathcal{S}(0)}=\gamma.\]
    \item \textbf{The Dynamical Component:}
    For any trajectory $\mathbf{t}$ that satisfies the dual convergence criteria (\cref{lem:convergence-criteria}) and generic vector $\mathbf{v}$, the dual system of the trajectory matrix isolates the initial state vector $\mathcal{S}(\mathbf{0})$ up to scaling (\cref{lem:recessive-trajectories}). That is, there exists a sequence $c_n\in\mathbb{R}\setminus\{0\}$ such that
    \[
    \lim_{n\to\infty}
    c_n\Phi_{\mathcal{M}_{\mathbf{0}, \mathbf{t}}}(n)^{-T}\mathbf{v}
    =
    \mathcal{S}(\mathbf{0}).
    \]
    We have proven that trajectories in $\mathcal{T}^*_3$ satisfy the dual convergence criteria, and it is conjectured that trajectories in $\mathcal{T}_r$ also satisfy it.
\end{enumerate}

\subsection{The Discovery Process}
The results presented in this paper did not emerge from a grand initial design, but rather from numerical experimentation.
While exploring the behavior of the $_2F_2$ hypergeometric function within the framework of Conservative Matrix Fields, we observed a surprising phenomenon:
the spontaneous emergence of Euler's constant $\gamma$. 

Investigating the source of this coincidence revealed that the phenomenon could be understood more naturally through the Meijer G-function, of which $_pF_q$ is a special case.
This shift in perspective suggested that the G-function itself provides the structural backbone for the CMF construction.
Following this lead, we utilized the PSLQ integer relation algorithm~\cite{ferguson1999analysis} to conduct a broad numerical search for linear combinations of G-function evaluations that sum to $\gamma$.
The patterns revealed by these experiments ultimately allowed us to reverse-engineer and rigorously establish \cref{prop:g-function-identity}.

\subsection{Future Work}

This work lays the foundation for a new perspective on constructing Apéry-like sequences for Euler's constant via symbolic identities and their associated Conservative Matrix Fields. Several promising directions for future research may include:

\begin{enumerate}
\item \textbf{\boldmath Extension to \(\gamma+\log(z)\).}
During our research, we also observed that the same construction appears to
approximate values of the form \(\gamma+\log(z)\) by evaluating the CMF at different values of $z$.
This suggests further connections to Rivoal's work \cite{rivoal2009rational} beyond the connection shown in \cref{subsec:rivoal-equivalence}.
A natural future direction is to generalize the Meijer \(G\)-function identity accordingly and clarify whether Rivoal's recurrences arise as a special case of the CMF framework.

\item \textbf{Beyond the Meijer G-function} \\
Building on our method of deriving infinite families of holonomic sequences converging to $\gamma$ via the Meijer G-function, a question that arises naturally is: can alternative parameterizations of the G-function also be used to construct holonomic sequences in a similar way? Moreover, could D-finite functions other than the Generalized Hypergeometric Function and the Meijer G-function be used for this purpose? Such investigations may yield novel sequences with potentially improved irrationality measures or uncover new symbolic approximations to constants beyond $\gamma$.
    
\item \textbf{Unifying Holonomic Sequences via CMFs} \\
Alongside the framework developed in this work for Euler's constant $\gamma$, recent work on $\pi$ \cite{raz2025euler} demonstrates the potential of interpreting known holonomic sequences as different trajectories within a CMF. This raises a broader open question: \emph{Can all Apéry-like approximations to mathematical constants be geometrically embedded as specific discrete trajectories within a D-finite Conservative Matrix Field?}

Addressing this could elevate isolated recurrence constructions into a unified symbolic framework, where the target constants emerge naturally from the dual convergence limits of the underlying CMF.

\item \textbf{Toward a Comprehensive Theory of Mathematical Constants} \\
As demonstrated in this work, CMFs encode entire \emph{families of holonomic sequences} to mathematical constants. While some theoretical foundations have been laid \cite{weinbaum2025conservative}, many structural properties remain supported only by numerical evidence. Developing a deeper theory of CMFs could help focus the search for new irrationality proofs that may be hiding within them.

\section*{Declaration of AI use}
During the development of the proof of \Cref{lem:asymptotic-extraction}, the first author used GPT-5.6 sol to study mathematical techniques outside his prior expertise. This interaction helped develop a variant of the proof, which was then written by the authors in \Cref{proof:asymptotic-extraction}.

\end{enumerate}

% \bibliographystyle{amsalpha}   % or amsplain
% \bibliography{library.bib}  
\printbibliography

\appendix

\section{Proof of~\texorpdfstring{\Cref{lem:asymptotic-extraction}}{Lemma~\ref{lem:asymptotic-extraction}}}
\label{proof:asymptotic-extraction}

We use the standard method of steepest descent, deforming the integration
contour through a simple saddle point along a path of constant phase
\cites[Section 2.4(iv)]{NIST:DLMF}[Section 4.10]{olver1997asymptotics}[Section 1.1]{paris2001asymptotics}.
For Mellin--Barnes integrals, the deformation must additionally respect the
Gamma-function pole families and produce no boundary contribution at infinity
\cites[Section 16.17]{NIST:DLMF}[Section 3.4]{paris2001asymptotics}.
For analogous saddle-point deformations of Mellin--Barnes loop integrals with
Gamma factors and large parameters, see
\cite[Sections 5.3.2--5.3.3]{paris2001asymptotics}.

In our setting, the phase and amplitude retain a dependence on the asymptotic parameter even after normalization, so the classical fixed-function formulation does not by itself provide the required uniform control.
We therefore use the parameter-uniform version of Laplace's method due to
Neuschel \cite[Theorem~3.1]{neuschel2012uniform} which gives a uniform asymptotic expansion
for a fixed contour when the phase and amplitude vary continuously with an
auxiliary parameter in a compact set:

\begin{theorem}
\label{thm:neuschel}
Let \(G\subset\mathbb C\) be a domain and \(S\subset\mathbb C\) a nonempty
compact set equipped with the subspace topology.
Let \(\mathcal P\) be a piecewise continuously differentiable non-closed
simple curve with finite starting point \(a\in G\) and finite or infinite
endpoint \(b\), such that \([a,b)_{\mathcal P}\subset G\).
Let
\[
p,q:G\times S\longrightarrow\mathbb C
\]
be continuous functions such that
\(
p_z:=p(\,\cdot\,,z)
\)
and
\(
q_z:=q(\,\cdot\,,z)
\)
are holomorphic on \(G\) for every \(z\in S\).

Suppose that, in a neighborhood of \(a\), they admit the convergent
expansions
\[
p_z(t)
=
p_z(a)
+
\sum_{\nu=0}^{\infty}
b_\nu(z)(t-a)^{\mu+\nu},
\qquad
\mu\in\mathbb N,
\qquad
b_0(z)\neq0,
\]
and
\[
q_z(t)
=
\sum_{\nu=0}^{\infty}
c_\nu(z)(t-a)^{\nu+\lambda-1},
\qquad
\lambda\in\mathbb N,
\qquad
c_0(z)\neq0.
\]

Assume further that, for every \(t\in(a,b)_{\mathcal P}\),
\[
\inf_{z\in S}
\Re\!\left(p_z(t)-p_z(a)\right)>0,
\]
and that this expression has a positive lower bound as \(t\to b\) along
\(\mathcal P\).

Finally, suppose that there exists \(N\in\mathbb N\) such that
\[
I(n;z)
:=
\int_{\mathcal P}
e^{-n p_z(t)}q_z(t)\,dt
\]
converges absolutely for all \(n\ge N\) and \(z\in S\), and that
\[
\int_{\mathcal P}
e^{-N\Re(p_z(t))}
|q_z(t)|\,|dt|
\]
is bounded as a function of \(z\in S\).

Then, as \(n\to\infty\),
\[
I(n;z)
\sim
e^{-n p_z(a)}
\sum_{\nu=0}^{\infty}
\Gamma\!\left(\frac{\nu+\lambda}{\mu}\right)
\frac{a_\nu(z)}{n^{(\nu+\lambda)/\mu}},
\]
uniformly with respect to \(z\in S\). The first coefficients are
\[
a_0(z)
=
\frac{c_0(z)}
{\mu\bigl(b_0(z)\bigr)^{\lambda/\mu}},
\]
and
\[
a_1(z)
=
\left(
\frac{c_1(z)}{\mu}
-
\frac{(\lambda+1)b_1(z)c_0(z)}
{\mu^2b_0(z)}
\right)
\frac{1}
{\bigl(b_0(z)\bigr)^{(\lambda+1)/\mu}}.
\]
The branches of powers with base \(b_0(z)\) are chosen as described in the first part of the proof \cite[Theorem~3.1]{neuschel2012uniform}.
\end{theorem}

\begin{proof}[Proof of \cref{lem:asymptotic-extraction}]

The proof has three main steps. We first identify the exact saddle points and
construct fixed limiting steepest-descent contours around their scaled limits.
We then pull these contours back to the original \(x\)-plane, obtaining
\(n\)-dependent Mellin--Barnes contours that pass through the exact saddles.
Finally, we apply the inverse normalization, which sends these contours back
to fixed paths with the saddle at the origin, and use the uniform
Laplace method stated in \cref{thm:neuschel} to extract the asymptotic growth of the three solutions.

\textbf{Step 1: Saddle points and limiting contours.}

We bring the integral representations of
$\mathbf I_1,\mathbf I_2,\mathbf I_3$ into the form
required by \cref{thm:neuschel} by identifying
the corresponding saddle points and integration contours.

Fix \(j\in\{1,2,3\}\), and write
\(\tau_1=2\pi i\), \(\tau_2=-2\pi i\), and \(\tau_3=0\).

On a simply connected pole-free domain, choose
branches of the logarithms and write
\(K_j(n,x)=e^{\Psi_j(x,n)}\), where
\[
\Psi_j(x,n)
=
3\log\Gamma(x)
+
\sum_{i=1}^2
\log\Gamma(\alpha_i n-x+1)
+
\tau_jx.
\]
After scaling \(x=n^{2/3}s\), the saddle equation
\(\partial_x\Psi_j(x,n)=0\) becomes
\[
F_{j,n}(s)
:=
3\psi(n^{2/3}s)
-
\sum_{i=1}^2
\psi(\alpha_i n-n^{2/3}s+1)
+
\tau_j
=
0.
\]
By \cref{lem:gamma-asymptotics}, as \(n\to\infty\),
with \(\alpha_1,\alpha_2\) fixed, uniformly for \(s\)
in compact subsets of the principal slit plane,
\[
F_{j,n}(s)
=
3\log s
-
\log(\alpha_1\alpha_2)
+
\tau_j
+
\mathcal O(n^{-1/3}).
\]
The limit as \(n\to\infty\) is therefore
\[
F_j(s)
:=
\lim_{n\to\infty}F_{j,n}(s)
=
3\log s-\log(\alpha_1\alpha_2)+\tau_j.
\]
It has a simple zero at \(c\rho_j\), since
\[
F_j(c\rho_j)=0,
\qquad
F_j'(c\rho_j)=\frac{3}{c\rho_j}\neq0.
\]

Choose a sufficiently small disk \(D_j\) centered at \(c\rho_j\), contained
in the principal slit plane, such that \(F_j\) has no other zero in
\(\overline{D_j}\). Since \(F_j\) is bounded away from zero on
\(\partial D_j\), the uniform estimate above gives, for all sufficiently
large \(n\),
\[
|F_{j,n}(s)-F_j(s)|<|F_j(s)|,
\qquad s\in\partial D_j.
\]
Rouch\'e's theorem therefore gives a unique simple zero
\(\sigma_{j,n}\in D_j\) of \(F_{j,n}\).
Since \(F_j'(c\rho_j)\neq0\), we have
\(|F_j(s)|\asymp|s-c\rho_j|\) near \(c\rho_j\).
Using \(F_{j,n}(\sigma_{j,n})=0\) and the uniform estimate above,
\[
\sigma_{j,n}
=
c\rho_j+\mathcal O(n^{-1/3}).
\]
Consequently, for all sufficiently large \(n\),
the phase \(\Psi_j(x,n)\) has exactly one saddle point
in \(n^{2/3}D_j\), located at
\[
x_{j,n}
=
n^{2/3}\sigma_{j,n}
=
c\rho_j n^{2/3}
+
\mathcal O(n^{1/3}).
\]

The limiting saddle equation is the derivative of the phase
\[
\phi_j(s)
=
s\left(
3\log\frac{s}{c}+\tau_j-3
\right),
\qquad
\phi_j'(s)=F_j(s),
\]
defined on the principal slit plane
\(\Omega=\mathbb C\setminus(-\infty,0]\).
Thus \(c\rho_j\) is the unique simple saddle of \(\phi_j\) in \(\Omega\).
Let \(\Gamma_j\) be the union of the two maximal
steepest-descent trajectories whose closures contain
the saddle \(c\rho_j\).
Away from the saddle, these trajectories follow
the negative gradient flow
\[
\frac{ds}{dt}=-\overline{\phi_j'(s)},
\]
along which
\[
\frac{d}{dt}\phi_j(s(t))
=
-|\phi_j'(s(t))|^2.
\]
Consequently, along either branch,
\(\operatorname{Im}(\phi_j)\) is constant and
\(\operatorname{Re}(\phi_j)\) strictly decreases
away from the saddle.

Neither branch can terminate at the origin. For \(j=1,2\),
\[
\operatorname{Im}(\phi_j(c\rho_j))
=
(-1)^{j+1}\frac{3\sqrt3}{2}c\neq0,
\]
whereas \(\phi_j(s)\to0\) as \(s\to0\). For \(j=3\),
\(\phi_3(c)=-3c\), while \(\operatorname{Re}(\phi_3)\) decreases along the
branch and tends to \(0\) at the origin. Nor can a branch meet the
logarithmic cut, since for \(r>0\),
\[
\phi_j(-r\pm i0)
=
3r\left(1+\log\frac{c}{r}\right)
-r(\tau_j\pm3\pi i).
\]
For \(j=3\), the imaginary part cannot equal the saddle phase. For
\(j=1,2\), equality is possible on only one bank and forces
\(r=qc\), where
\(q=3\sqrt3/(2\pi)\in(1/2,1)\); there,
\[
\operatorname{Re}\left(\phi_j(-qc\pm i0)\right)
=
3cq(1-\log q)
>
\frac{3c}{2}
=
\operatorname{Re}\left(\phi_j(c\rho_j)\right),
\]
contradicting descent.
Since there are no other critical points in \(\Omega\), by standard continuation of the gradient-flow trajectories, both maximal branches extend to infinity.

To determine their directions, write \(s=re^{i\theta}\), with
\(-\pi<\theta<\pi\). The constant-phase condition along \(\Gamma_j\) gives
\[
\frac{\operatorname{Im}(\phi_j(c\rho_j))}{r}
=
\left(3\log\frac{r}{c}-3\right)\sin\theta
+
\left(3\theta+\operatorname{Im}(\tau_j)\right)\cos\theta.
\]
As \(r\to\infty\), this implies
\(\sin\theta=\mathcal O((\log r)^{-1})\), so every end approaches either
the positive or the negative real direction. The positive direction is
impossible because there
\(\operatorname{Re}(\phi_j(s))\to+\infty\), whereas
\(\operatorname{Re}(\phi_j)\) decreases along a descent branch.
Hence both ends approach the negative real direction.

For sufficiently large \(r\), the level equation has one solution near each
bank of the negative real axis. Writing
\(\theta=\pi-\varepsilon\) on the upper bank and
\(\theta=-\pi+\varepsilon\) on the lower bank gives
\[
\varepsilon
=
\frac{3\pi+\operatorname{Im}(\tau_j)}
{3\log(r/c)}
+
\mathcal O((\log r)^{-2}),
\qquad
\varepsilon
=
\frac{3\pi-\operatorname{Im}(\tau_j)}
{3\log(r/c)}
+
\mathcal O((\log r)^{-2}),
\]
respectively. Since \(3\pi\pm\operatorname{Im}(\tau_j)>0\), the two ends
approach the negative real direction from opposite sides, and
\[
|\operatorname{Im}(s)|
=
r|\sin\theta|
\asymp
\frac{r}{\log r}
\longrightarrow\infty.
\]
Thus the tails of \(\Gamma_j\) move arbitrarily far from the negative real
axis. Since the remaining compact part does not meet the cut,
\[
\operatorname{dist}\!\left(\Gamma_j,(-\infty,0]\right)>0.
\]

\textbf{Step 2: Mellin--Barnes contours.}

We now pull the limiting contour back to the original \(x\)-plane. Define
\[
C_{j,n}
=
n^{2/3}\left(
\Gamma_j+\sigma_{j,n}-c\rho_j
\right).
\]
Since \(\sigma_{j,n}-c\rho_j=\mathcal O(n^{-1/3})\) and
\(\operatorname{dist}(\Gamma_j,(-\infty,0])>0\), the translated contour
\(\Gamma_j+\sigma_{j,n}-c\rho_j\) remains disjoint from the negative real
axis for all sufficiently large \(n\). Hence \(C_{j,n}\) avoids the left
pole family \(0,-1,-2,\ldots\). The right poles are
\[
x=\alpha_i n+1+k,
\qquad i=1,2,\quad k\ge0.
\]
After scaling by \(n^{-2/3}\), these lie at
\[
\alpha_i n^{1/3}+(1+k)n^{-2/3},
\]
while \(\Gamma_j\), and hence its small translate, is bounded to the right.
Thus all right poles lie to the right of \(C_{j,n}\) for sufficiently large
\(n\). Moreover, by Step~1, the two ends of \(C_{j,n}\) tend to negative
infinity from opposite sides of the real axis. Orienting \(C_{j,n}\)
consistently with \(L'\), it therefore has the prescribed right-loop topology and
separates the two pole families.

It remains to check convergence at infinity. By the reflection formula,
\[
K_j(n,x)
=
\frac{
\pi^3
\Gamma(\alpha_1n-x+1)
\Gamma(\alpha_2n-x+1)
}{
\sin^3(\pi x)\Gamma(1-x)^3
}
e^{\tau_jx}.
\]
Along either tail of \(C_{j,n}\), the quantities
\(1-x\) and \(\alpha_i n-x+1\) tend to infinity in a fixed sector about the
positive real axis. Hence \cref{lem:gamma-asymptotics} gives
\[
\log\left|
\frac{
\Gamma(\alpha_1n-x+1)
\Gamma(\alpha_2n-x+1)
}{
\Gamma(1-x)^3
}
\right|
=
-\operatorname{Re}\!\left((1-x)\log(1-x)-(1-x)\right)
+\mathcal O_n(\log|x|),
\]
which tends to \(-\infty\) like \(-|x|\log|x|\). The remaining factor also
decays exponentially: on the upper tail,
\[
\left|
\frac{e^{\tau_jx}}{\sin^3(\pi x)}
\right|
=
\exp\left(
-\bigl(3\pi+\operatorname{Im}(\tau_j)\bigr)\operatorname{Im}(x)
+\mathcal O(1)
\right),
\]
while on the lower tail it is
\[
\exp\left(
-\bigl(3\pi-\operatorname{Im}(\tau_j)\bigr)|\operatorname{Im}(x)|
+\mathcal O(1)
\right).
\]
Since \(3\pi\pm\operatorname{Im}(\tau_j)>0\), the integrals over \(C_{j,n}\)
converge absolutely; the polynomial factors
\(1,n-x,(n-x)^2\) do not affect this conclusion.

Thus \(C_{j,n}\) separates the two pole families, has the prescribed
right-loop behavior at infinity, and the corresponding integrals converge
absolutely. Hence \(C_{j,n}\) satisfies the contour conditions of
\cref{def:meijer-g} and is an admissible choice for \(L'\).
In particular, we may evaluate \(\mathbf I_j(n)\) on \(C_{j,n}\), whose defining advantage is that it passes through the exact saddle \(x=n^{2/3}\sigma_{j,n}\).

\textbf{Step 3: Uniform saddle asymptotics.}

We now normalize the contours back to fixed paths. Set
\[
P_j:=\Gamma_j-c\rho_j,
\qquad
x=n^{2/3}(\sigma_{j,n}+t).
\]
By construction, this maps \(P_j\) onto \(C_{j,n}\), with the exact saddle
corresponding to \(t=0\). Define
\[
\Phi_{j,n}(t)
:=
n^{-2/3}\left(
\Psi_j(n^{2/3}(\sigma_{j,n}+t),n)
-
\Psi_j(n^{2/3}\sigma_{j,n},n)
\right)
\]
and
\[
q_{j,n}(t)
:=
1-n^{-1/3}(\sigma_{j,n}+t).
\]
Since \(n-x=nq_{j,n}(t)\),
\[
\mathbf I_j(n)^{T}
=
\frac{n^{2/3}}{2\pi i}
e^{\Psi_j(n^{2/3}\sigma_{j,n},n)}
\int_{P_j}
e^{n^{2/3}\Phi_{j,n}(t)}
\left(
1,\,
nq_{j,n}(t),\,
n^2q_{j,n}(t)^2
\right)\,dt.
\]

We apply \cref{thm:neuschel} componentwise to the two halves of
\(P_j\) meeting at \(t=0\). Let
\[
S_N=\{0\}\cup\{n^{-1/3}:n\ge N\}.
\]
For each \(n\ge N\), let \(z_n=n^{-1/3}\) and set
\[
p_{z_n}(t):=-\Phi_{j,n}(t),
\qquad
q_{z_n}(t):=q_{j,n}(t).
\]
At \(z=0\) set
\[
p_0(t)
=
\phi_j(c\rho_j)-\phi_j(c\rho_j+t),
\qquad
q_0(t)=1.
\]
By \(\sigma_{j,n}=c\rho_j+\mathcal O(n^{-1/3})\) and
\cref{lem:gamma-asymptotics}, these functions converge locally uniformly,
together with the required derivatives, as \(z\to0\).

Step~1 gives
\[
\operatorname{dist}\!\left(
P_j,-c\rho_j+(-\infty,0]
\right)>0.
\]
Thus the normalized left poles remain uniformly separated from \(P_j\), whereas the normalized right poles escape to \(+\infty\).
After increasing \(N\) if necessary, the phases and amplitudes are therefore holomorphic on a common simply connected neighborhood \(G_j\) of \(P_j\) for all
\(z\in S_N\).

At the saddle, \(p_z(0)=p_z'(0)=0\), while for \(z=z_n=n^{-1/3}\),
\[
p_z''(0)
=
-n^{2/3}\Psi_j''(n^{2/3}\sigma_{j,n},n)
\longrightarrow
-\frac{3}{c\rho_j}\neq0, \quad
q_z(0)\to1
\]
Hence the local expansion in Neuschel's theorem has \(\mu=2\) and \(\lambda=1\). 
Moreover, \(\operatorname{Re}(p_0(t))>0\) for \(t\in P_j\setminus\{0\}\), since \(P_j\) is the translated steepest-descent contour of \(\phi_j\).
Since \(P_j\) leaves the origin in the steepest-descent directions and \(p_z''(0)\to-3/(c\rho_j)\neq0\), after increasing \(N\) if necessary there exist \(a,\delta>0\) such that
\[
\operatorname{Re}(p_z(t))
\ge
a|t|^2,
\qquad
0<|t|\le\delta,\quad t\in P_j,
\]
uniformly for \(z\in S_N\).
On every compact subarc \(\delta\le|t|\le R\), local uniform convergence to \(p_0\) and the strict positivity of \(\operatorname{Re}(p_0)\) give a uniform positive lower
bound.

On the two tails of \(P_j\), step 1 shows that
\(t\) approaches the negative real direction, uniformly
away from the logarithmic cut. In particular, for
sufficiently large \(|t|\), the arguments of
\(1-x\) and \(\alpha_i n-x+1\), where
\[
x=n^{2/3}(\sigma_{j,n}+t),
\]
remain in a fixed sector about the positive real axis.

Using the reflection formula and Stirling's estimates
in this sector, the normalized phase satisfies
\[
\operatorname{Re}p_{z_n}(t)
\ge
\kappa |t|
\left(
3\log|t|
-
2\log\left(1+\frac{|t|}{n^{1/3}}\right)
\right)
-C|t|-C,
\]
for sufficiently large \(|t|\), with constants
\(\kappa>0\) and \(C>0\) independent of \(n\).

The three reflected Gamma factors contribute
the leading decay proportional to
\(3|t|\log|t|\), whereas the two remaining Gamma
factors contribute at most
\(2|t|\log(1+|t|/n^{1/3})\).

Since \(n^{1/3}\ge1\),
\[
3\log|t|
-
2\log\left(1+\frac{|t|}{n^{1/3}}\right)
\ge
3\log|t|-2\log(1+|t|)
=
\log|t|+\mathcal O(1).
\]
Consequently, there exist constants \(A,B>0\), independent of \(n\),
such that
\[
\operatorname{Re}p_{z_n}(t)
\ge
A|t|\log(2+|t|)-B.
\]

The same estimate holds for \(p_0\) by the
corresponding tail asymptotics of the limiting
phase \(\phi_j\). Since the amplitudes
\(1,q_z,q_z^2\) grow at most polynomially,
the required uniform integrability follows.

The restriction in Neuschel's statement to an integer Laplace parameter is inessential, so we apply the theorem with parameter \(n^{2/3}\).
With the orientations inherited from \(P_j\), the two half-contour
expansions combine into the nonzero simple-saddle contribution.
Consequently, there exist constants \(A_j\neq0\) such that, defining the complex scalar amplitude
\[
\mathcal A_j(n)
:=
A_j n^{1/3}
e^{\Psi_j(n^{2/3}\sigma_{j,n},n)},
\]
we have
\[
\mathbf I_j(n)^{T}
=
\mathcal A_j(n)
\left(
1+o(1),\,
n(1+o(1)),\,
n^2(1+o(1))
\right).
\]
Thus the three components share the same complex exponential amplitude,
while the polynomial factors determine the asymptotic shape of the state
vector.

By \cref{lem:gamma-asymptotics}, uniformly for \(s\) in a bounded neighborhood of \(c\rho_j\),
\[
\Psi_j(n^{2/3}s,n)
=
\log H(n)
+
n^{2/3}\phi_j(s)
+
\mathcal O(n^{1/3}).
\]
Since \(\sigma_{j,n}=c\rho_j+\mathcal O(n^{-1/3})\) and
\(\phi_j'(c\rho_j)=0\),
\[
n^{2/3}\phi_j(\sigma_{j,n})
=
n^{2/3}\phi_j(c\rho_j)+\mathcal O(1)
=
-3c\rho_jn^{2/3}+\mathcal O(1).
\]
Hence
\[
\Psi_j(n^{2/3}\sigma_{j,n},n)
=
\log H(n)
-
3c\rho_jn^{2/3}
+
\mathcal O(n^{1/3}).
\]
Taking absolute values of the amplitude and using
\(\|(1,n,n^2)\|=\Theta(n^2)\), we obtain
\[
\|\mathbf I_j(n)\|
=
H(n)
\exp\left(
-3c\,\operatorname{Re}(\rho_j)n^{2/3}
+
\mathcal O(n^{1/3})
\right),
\]
as claimed.

It remains to prove linear independence.
\(\mathbf I_3(n)\) has a strictly smaller exponential scale than the first two.
Expanding the saddle equation and the Stirling formula one order further
gives, for the two dominant amplitudes,
\[
\frac{
\mathcal A_1(n+1)/\mathcal A_2(n+1)
}{
\mathcal A_1(n)/\mathcal A_2(n)
}
=
\exp\left(
-3c(\rho_1-\rho_2)
\left((n+1)^{2/3}-n^{2/3}\right)
+
\mathcal O(n^{-2/3})
\right).
\]
Since
\[
\rho_1-\rho_2=-i\sqrt3,
\qquad
(n+1)^{2/3}-n^{2/3}
=
\frac23n^{-1/3}+\mathcal O(n^{-4/3}),
\]
we obtain
\[
\frac{
\mathcal A_1(n+1)/\mathcal A_2(n+1)
}{
\mathcal A_1(n)/\mathcal A_2(n)
}
=
\exp\left(
2i\sqrt3\,c\,n^{-1/3}
+
\mathcal O(n^{-2/3})
\right).
\]

Thus an unwrapped argument of
\(\mathcal A_1(n)/\mathcal A_2(n)\) has eventually positive increments
asymptotic to \(2\sqrt3\,c\,n^{-1/3}\). Since
\[
\sum_{n=1}^{\infty}n^{-1/3}=\infty,
\]
the relative amplitude cannot converge to a nonzero constant.

Suppose now that
\[
\lambda_1\mathbf I_1
+
\lambda_2\mathbf I_2
+
\lambda_3\mathbf I_3
=0
\]
for constants \(\lambda_1,\lambda_2,\lambda_3\). By the exponential
separation of \(\mathbf I_3\), any nontrivial relation must have
\(\lambda_1\lambda_2\neq0\). Taking the first component and using the
vector asymptotics above gives
\[
\frac{\mathcal A_1(n)}{\mathcal A_2(n)}
=
-\frac{\lambda_2}{\lambda_1}
+
o(1),
\]
a contradiction. Hence the three functions are linearly independent.

\end{proof}

\end{document}